\documentclass[11pt]{amsart}  
\usepackage[margin=1in]{geometry}

\usepackage{MnSymbol}
\usepackage{colonequals}
\usepackage{amsthm}
\usepackage{mathrsfs}
\usepackage{tikz-cd}
\usepackage{comment}
\usepackage{mathtools}
\usepackage{manfnt}
\usepackage{colonequals}

\makeatletter
\newcommand{\oset}[3][0ex]{%
  \mathrel{\mathop{#3}\limits^{
    \vbox to#1{\kern-2\ex@
    \hbox{$\scriptstyle#2$}\vss}}}}
\makeatother

\usepackage{hyperref}
\hypersetup{
pdflang={en-US},
pdftitle={Stability of natural bundles on curves},
pdfauthor={Izzet Coskun, Eric Larson, and Isabel Vogt},
pdfsubject={Algebraic Geometry},
pdfkeywords={estricted tangent bundle, the normal bundle, stability, Brill-Noether curves}
}

\newcommand{\PP}{\mathbb{P}}

\newcommand{\OO}{\mathcal{O}}
\newcommand{\Hom}{\operatorname{Hom}}
\newcommand{\End}{\operatorname{End}}
\newcommand{\Ext}{\operatorname{Ext}}
\newcommand{\Gal}{\operatorname{Gal}}
\newcommand{\Pic}{\operatorname{Pic}}
\newcommand{\HH}{\mathcal{H}}
\newcommand{\Coh}{\operatorname{Coh}}

\newcommand{\Mor}{\operatorname{Mor}}
\newcommand{\Hilb}{\operatorname{Hilb}}
\newcommand{\codim}{\operatorname{codim}}

\newcommand{\defi}[1]{\textsf{#1}} 
\newcommand{\Tsch}{\operatorname{Tsch}}
\newcommand{\DTs}{\operatorname{Tsch^\vee}}
\newcommand{\et}{\text{\'et}}
\newcommand{\pr}{\text{pr}}
\newcommand{\rk}{\operatorname{rk}}

\newcommand{\leqor}{\underset{{\scriptscriptstyle (}-{\scriptscriptstyle )}}{<}}

\newcommand{\posmod}{\oset{+}{\rightarrow}}

\newcommand{\negmod}{\oset{-}{\rightarrow}}
\newcommand{\posmodalong}{\oset{+}{\leadsto}}

\newtheorem{thm}{Theorem}[section]
\newtheorem{lem}[thm]{Lemma}
\newtheorem{prop}[thm]{Proposition}
\newtheorem{cor}[thm]{Corollary}
\newtheorem{conj}[thm]{Conjecture}

\theoremstyle{definition}
\newtheorem{defin}[thm]{Definition}

\theoremstyle{remark}
\newtheorem{rem}[thm]{Remark}
\newtheorem{example}[thm]{Example}
\newtheorem{question}[thm]{Question}
\newtheorem{problem}[thm]{Problem}

\title{Stability of natural bundles on curves}

\author{Izzet Coskun}
\address{Department of Mathematics, Statistics, and CS \\
University of Illinois at Chicago, Chicago IL 60607}
\email{icoskun@uic.edu}
\author{Eric Larson}
\address{Department of Mathematics, Brown University}
\email{elarson3@gmail.com}
\author{Isabel Vogt}
\address{Department of Mathematics, Brown University}
\email{ivogt.math@gmail.com}

\thanks{During the preparation of this article, I.C.\ was supported
by  NSF grant DMS-2200684 and a Simons Foundation Travel Award,  E.L. was supported by
NSF  grant DMS-2440719, and I.V. was supported by NSF grant DMS-2338345.}

\begin{document}

\begin{abstract}
    In this paper, we survey recent developments concerning the stability of naturally defined bundles on curves that play a central role in the deformation theory of the curve. 
\end{abstract}
\maketitle
\section{Introduction}\label{sec-intro}

Given $X \subset Y$ two smooth projective varieties, there are several natural bundles that govern the deformation theory of $X$ in $Y$ and play a fundamental role in both geometry and arithmetic. In this survey, we will focus on the two most important examples of such bundles:  The restricted tangent bundle $TY|_X$ and the normal bundle $N_{X/Y}$ defined by the exact sequence
\begin{equation}\label{eq:normal}
    0 \to TX \to TY|_X \to N_{X/Y} \to 0.
\end{equation}
We will survey recent results establishing the stability of $TY|_X$ and $N_{X/Y}$, focusing primarily on the case where $X$ is a curve and $Y= \PP^r$.

\subsubsection*{The space of morphisms} Let $\Mor(X,Y)$ denote the scheme parameterizing morphisms $f\colon X \to Y$. This scheme plays a central role in enumerative geometry, Mori theory, the study of Fano varieties, and questions of hyperbolicity. The space of first-order deformations of the morphism $f$ is given by $H^0(X, f^*(TY))$, while the obstructions to lifting such deformations lie in $H^1(X, f^*(TY))$ (see \cite[Proposition 24.9]{hartshorne}). In particular, if $H^1(X, f^*(TY))=0$, then the deformations are unobstructed and $\Mor(X,Y)$ is smooth at $f$ of the expected dimension $h^0(X, f^*(TY))$. When $f$ is an embedding, $f^*(TY)$ coincides with the restricted tangent bundle $TY|_X$.

\subsubsection*{The Hilbert scheme} A similar description holds for embedded deformations. Let $\Hilb_{p}(Y)$ denote the Hilbert scheme parameterizing subschemes of $Y$ with fixed Hilbert polynomial $p$. The Zariski tangent space to $\Hilb_{p}(Y)$ at a point $[X]$ corresponding to $X \subset Y$ is naturally identified with $H^0(X, N_{X/Y})$. When $H^1(X, N_{X/Y})=0$, then $\Hilb_{p}(Y)$ is smooth at $[X]$ and has the expected dimension $h^0(X, N_{X/Y})$ (see \cite[Theorem 1.1]{hartshorne}). 

Consequently, understanding the cohomology of the restricted tangent bundle and the normal bundle lies at the heart of deformation and moduli theory, and  motivates the main question of this survey.

\begin{question}
    When are $TY|_X$ and $N_{X/Y}$ (semi)stable bundles?
\end{question}

Stable bundles enjoy good metric and cohomological properties and are the building blocks of all bundles via the Harder--Narasimhan and H\"{o}lder filtrations. 

Let $X$ be a smooth projective curve of genus $g$. The stack $\Coh(r,d)$ of coherent sheaves of rank $r$ and degree $d$ on $X$ is irreducible (see \cite{hoffman}). When $g \geq 1$, the locus of semistable bundles forms a nonempty open substack of $\Coh(r,d)$. Moreover, when  $g \geq 2$, the locus of stable bundles also forms a nonempty open substack (see \cite[\S3]{chnsurvey} for a more detailed survey). Consequently, in the absence of evident geometric obstructions, it is reasonable to expect that a general, naturally defined bundle is semistable when $g=1$ and stable when $g \geq 2$. In this survey, we illustrate both instances of this philosophy and examples where geometric obstructions arise.

In \S \ref{sec-rational}, we  begin by discussing $T\PP^r|_X$ and $N_{X/\PP^r}$ when $X$ is a rational curve. Since vector bundles on rational curves split as direct sums of line bundles, this case is considerably more concrete than the case of curves of higher genus, and these bundles can be computed explicitly by linear algebra. We will discuss a number of representative examples, such as the case of monomial curves. Since rational curves play a central role in the study of Fano varieties and questions of rational connectedness, we briefly digress to discuss $TY|_X$ and $N_{X/Y}$ when $X$ is a rational curve and $Y$ is a Grassmannian, or more generally a homogeneous variety or a Fano complete intersection in a homogeneous variety.

In \S \ref{sec-elliptic}, we will survey recent developments concerning $T \PP^r|_X$ and $N_{X/\PP^r}$ when $X$ is a curve of genus 1. Bundles on curves of genus 1 have been classified by Atiyah \cite{Atiyah}. Combining Atiyah's explicit description with the fact that genus 1 curves have the structure of a group, we give quick proofs that both $T\PP^r|_X$ and $N_{X/\PP^r}$ are semistable for general nondegenerate genus 1 curves.

For curves of higher genus, such explicit descriptions of stable bundles are not available. Instead, much of the progress in this setting has come from specialization techniques. The basic idea is to specialize a curve of genus $g$ and degree $d$ to a nodal union of curves of lower genus and degree. This approach is particularly effective for the restricted tangent bundle $T\PP^r|_X$, since when one specializes $X$ to a reducible nodal curve $X_1 \cup X_2$, the restricted tangent bundle $T\PP^r|_X$ specializes to the restricted tangent bundle $T\PP^r|_{X_i}$ on each component $X_i$. The most comprehensive result in this direction is due to Farkas and E.\ Larson \cite{farkaslarson}, who showed that $T\PP^r|_X$ is semistable if $X$ is a general Brill--Noether curve of genus at least $1$. Moreover, $T\PP^r|_X$ is in fact stable if $X$ is a general Brill--Noether curve of genus at least $2$, except when $X$ is a genus 2 curve of degree $2r$ in $\PP^r$.

The normal bundle presents substantially greater difficulties. Although the normal bundles also form a flat family when $X$ specializes to a nodal reducible curve $X_1 \cup X_2$, the limit is not simply the union of the normal bundles, but rather certain elementary modifications of the normal bundles. This complication makes the study of the normal bundle significantly more subtle. Furthermore, there are examples in low degree and low genus for which $N_{X/\PP^r}$ fails to be semistable. Hence, any inductive approach must avoid these counterexamples. In \S \ref{sec-highgenus}, we  survey the current state of knowledge on the stability of $N_{X/\PP^r}$ when $X$ is a general Brill--Noether curve of higher genus.

Finally, in \S \ref{sec-covers}, we broaden our perspective and discuss the stability of vector bundles naturally associated with finite covers of curves. Given a finite cover of curves $f\colon X \to Y$,  one can consider the associated Tschirnhausen bundle, as well as  the pushforwards of general bundles on $X$. These bundles arise naturally in the study of theta divisors on the moduli space of bundles, branched covers of curves, and Hurwitz spaces.  Their stability properties have attracted significant recent attention \cite{deopurkarpatel, landesmanlitt, clvtschirnhausen, clvbeauville}.  The final section surveys recent results on the stability of these bundles.

\subsection*{Acknowledgments} We would like to thank Atanas Atanasov, Arnaud Beauville, Lawrence Ein, Gavril Farkas, Henry Fontana, Joe Harris, Jack Huizenga, Eric Jovinelly, Hannah Larson, Aaron Landesman, Daniel Litt, Angelo Lopez, Sayanta Mandal, Lucas Mioranci, Eric Riedl, Geoff Smith, Jason Starr, Ravi Vakil, and David Yang for many invaluable discussions on the topics discussed in this survey.

\section{Preliminaries}\label{sec-prelim}
In this section, we summarize basic definitions and foundational results that will play a role in our subsequent discussion.

\subsection{Stability}
We begin by recalling the definition and basic properties of semistable vector bundles on curves. For further details and proofs, we refer the reader to \cite{lepotierbook}.

\begin{defin}\label{def-stability}
    Let $(X,H)$ be a polarized variety of dimension \(n\). Then the \defi{$H$-slope $\mu_H$} of a torsion-free coherent sheaf $V$ on $X$ is defined by 
    \[\mu_H(V) = \frac{c_1(V) \cdot H^{n-1}}{\rk(V) H^n},\] where $\rk(V)$ is the rank of $V$ and $c_1(V)$ is the first Chern class of $V$. 
    The sheaf  $V$ is \defi{$\mu_H$-(semi)stable} if all proper subsheaves of smaller rank $W \subset V$ satisfy \[ \mu_H(W) \leqor \mu_H(V).\]  
\end{defin}

\begin{rem}\label{rem:sm_curve}
    When $X$ is a smooth projective curve, every torsion-free sheaf is locally free, i.e., a vector bundle.  Moreover, if \(E\) is a vector bundle and \(F \subset E\) is a subsheaf, let \(\bar{F}\) be the kernel of the map \(E \to (E/F)/(E/F)_{\text{tors}}\).  Then \(\rk \bar{F} = \rk F\) and \(\deg \bar{F} \geq \deg F\) because \(F \subset \bar{F}\).  To determine the (semi)stability of \(E\), it therefore suffices to only consider subsheaves that are simultaneously locally free and locally cofree, i.e.,  vector subbundles.
\end{rem}

\begin{rem}
    For higher-dimensional varieties, Definition \ref{def-stability} depends on the choice of polarization since the degree is computed in terms of $H$. In contrast, for curves the degree is intrinsic and the polarization $H$ does not play a role, up to overall rescaling. Since  this survey focuses mainly on curves, we will not introduce alternative notions of stability such as Gieseker stability, Matsuki-Wentworth stability, or Bridgeland stability. We refer the reader to the surveys \cite{chgokova, chnsurvey, huybrechtslehn} for detailed treatments of stability on higher-dimensional varieties.
\end{rem}

Every torsion-free sheaf $V$ on $X$ admits a unique filtration, called the \defi{Harder--Narasimhan (HN) filtration}
\[0 \subset E_1 \subset \cdots \subset E_n=V,\] such that the successive quotients $F_i \colonequals E_i/E_{i-1}$ are semistable and 
\[\mu_H(F_1) > \cdots > \mu_H(F_n).\] Furthermore, every semistable sheaf admits a further filtration, called the \defi{H\"{o}lder} filtration, so that the successive quotients are stable. While the H\"{o}lder filtration is not unique, the associated graded object is unique. Two sheaves are called \defi{$S$-equivalent} if their associated graded objects for the H\"{o}lder filtration are isomorphic.

Assume that $X$ is a smooth, projective curve of genus $g$. In general, the stack of vector bundles of fixed rank $r$ and degree $d$ on $X$ is not bounded. For example, if \(L\) is a degree \(1\) line bundle on \(X\), consider the family of bundles of rank 2 and degree \(d\)  \[V_a \colonequals L^{\otimes (-a)} \oplus L^{\otimes (d+a)}.\]  If $d \geq 2g-2$ and $a > 0$, then by Riemann-Roch we have $h^0(X, V_a)= d+a-g+1$. In particular, by letting $a$ tend to infinity, we can make both $h^0(X, V_a)$ and $h^1(X, V_a)$ arbitrarily large. Hence, the bundles $V_a$ do not belong in a bounded family. 

In contrast, semistability imposes strong cohomological constraints. For example, if $V$ is a semistable bundle on $X$ of slope $\mu > 2g-2$, then $V$ has no higher cohomology. Indeed, by Serre duality, \[h^1(X, V)= h^0(X, V^{\vee} \otimes \omega_X),\] where $\omega_X$ denotes the canonical sheaf of $X$. The bundle $V^{\vee} \otimes \omega_X$ is semistable and has slope less than 0, therefore, it admits no nonzero global sections, since a nonzero section would generate a subsheaf of nonnegative degree destabilizing $V^{\vee} \otimes \omega_X$. Hence, $h^1(X, V)=0$. This example illustrates that semistable bundles enjoy better cohomological properties than arbitrary bundles.

More importantly, there is an irreducible, projective moduli space parameterizing $S$-equivalence classes of semistable bundles of rank $r$ and degree $d$ on on a smooth projective curve $X$; for a modern construction, see \cite{abblt22}.

\subsection{The invariants}

Let $X \subset \PP^r$ be a curve of degree $d$ and genus $g$.  Restricting the Euler sequence 
\[0 \to \OO_{\PP^r} \to \OO_{\PP^r}(1)^{\oplus (r+1)} \to T\PP^r \to 0\]
to $X$, we obtain
\[ \rk(T \PP^r|_X)= r \quad \deg(T \PP^r|_X) = (r+1)d. \]
The defining exact sequence~\eqref{eq:normal} for $N_{X/\PP^r}$ then yields
\[\rk(N_{X/\PP^r}) = r-1 \quad \deg(N_{X/\PP^r})= (r+1)d + 2g-2. \]

\subsection{Elementary modifications and normal bundles of nodal curves} Let $X$ be a projective variety and let $V$ be a vector bundle on $X$. Let $D\subset X$ be a Cartier divisor and let $W \subset V|_D$ be a subbundle of the restriction of $V$ to $D$. Then  \defi{the negative elementary modification} $V[D \negmod W]$ of $V$ along $D$ towards  $W$ is defined by the exact sequence
$$0 \to V[D \negmod W] \to V \to V|_D/W \to 0.$$ The \defi{positive elementary modification} of $V$ along $D$ towards $W$ is defined by 
\[V[D \posmod W]\colonequals V[D \negmod W](D).\] We refer the reader to \cite[\S 2-6]{aly} for a detailed discussion of the properties of modifications. The  modification $V[D \posmod W]$ is naturally isomorphic to $V$ outside of $D$, so one can easily define multiple modifications of $V$ along divisors with disjoint support, which we will denote by \[V[D_1 \posmod W_1] \cdots [D_j\posmod W_j].\] When the supports of the divisors meet, then defining multiple modifications is more subtle. When the bundles $W_i$ extend to bundles in a neighborhood of $D_i$, one can define multiple modifications in some contexts. For example, multiple modifications always exist when \(X\) is a smooth curve.  We refer the reader to \cite{aly, interpolation} for details.

\subsection{Normal bundles of nodal curves} Let $X= X_1 \cup X_2 \subset \PP^r$ be a connected nodal curve. Then at a node $p_i$ of $C$,  the tangent line $T_{p_i}X_2$ defines a normal direction in $N_{X_1/\PP^r}|_{p_i}$. We write $N_{X_1/\PP^r}[p_i\posmodalong X_2]$ for the positive elementary modification $N_{X_1/\PP^r}[p_i\posmod T_{p_i}X_2]$. We then have the following fundamental observation of Hartshorne and Hirschowitz. 

\begin{lem}\cite[Corollary 3.2]{hartshornehirschowitz}\label{lem-harthir}
Let $X=X_1\cup X_2 \in \PP^r$ be a connected nodal curve. Let $X_1 \cap X_2= \{ p_1, \dots, p_j\}$. Then 
$$N_{X/\PP^r}|_{X_1} \cong N_{X_1/\PP^r}[p_1 \posmodalong X_2] \cdots [p_j \posmodalong X_2 ]$$
\end{lem}

Lemma~\ref{lem-harthir} allows us to study the stability of normal bundles inductively. However, carrying out this induction requires considering modifications of the normal bundle.

\subsection{Stability on nodal curves} When \(X\) is a \emph{smooth} projective curve, we observed in Remark~\ref{rem:sm_curve} that it suffices to consider only vector subbundles when proving (semi)stability. In order to use specialization, it will be convenient to provide an alternative notion of stability on nodal curves that similarly uses only vector subbundles and is also open in families. We refer the reader to  \cite[\S 2]{clv} for details. Let \(X\) be a connected nodal curve and let $\nu \colon \tilde{X} \to X$  be its normalization.  For any node \(p\) of \(X\), let \(\tilde{p}_1\) and \(\tilde{p}_2\) be the two points of \(\tilde{X}\) over \(p\). Given a vector bundle \(V\) on \(X\), the fibers of the pullback \(\nu^*V\) to \(\tilde{X}\) over \(\tilde{p}_1\) and \(\tilde{p}_2\) are naturally identified.  Given a subbundle \(F \subseteq \nu^*V\), we can thus compare \(F|_{\tilde{p}_1}\) and \(F|_{\tilde{p}_2}\) inside \(\nu^*V|_{\tilde{p}_1} \simeq \nu^*V|_{\tilde{p}_2}\).

\begin{defin}
Let \(V\) be a vector bundle on a connected nodal curve \(X\).  
For a subbundle \(F \subset \nu^*V\),  the \defi{adjusted slope \(\mu^{\text{adj}}_X\)} is given by  
\[\mu^{\text{adj}}_X(F) \colonequals \mu(F) - \frac{1}{\rk{F}} \sum_{p \in X_{\text{sing}}} \codim_{F} \left(F|_{\tilde{p}_1}\cap F|_{\tilde{p}_2} \right),\]
where \(\codim_F \left(F|_{\tilde{p}_1}\cap F|_{\tilde{p}_2} \right)\) refers to the codimension
of the intersection in either \(F|_{\tilde{p}_1}\) or \(F|_{\tilde{p}_2}\). 
We say that \(V\) is \defi{(semi)stable} if, for all subbundles \(F \subset \nu^*V\),
\[\mu^{\text{adj}}(F) \leqor \mu(\nu^*V) = \mu(V). \]
\end{defin}

\noindent
This notion of semistability specializes well in families of nodal curves.

\begin{prop} \cite[Proposition 2.3]{clv} \label{prop:stab-open}
Let \(\mathcal{X} \to \Delta\) be a family of connected nodal curves over
the spectrum of a discrete valuation ring,
and \(\mathcal{V}\) be a vector bundle on \(\mathcal{X}\).
If the special fiber \(\mathcal{V}_0 = \mathcal{V}|_0\) is (semi)stable,
then the general fiber \(\mathcal{V}^* = \mathcal{V}|_{\Delta^*}\) is also (semi)stable.
\end{prop}

\begin{lem} \cite[Lemma 4.1]{clv}\label{lem:naive}
Suppose that \(X = X_1 \cup X_2\) is a reducible nodal curve and \(V\) is a vector bundle on \(X\) such that \(V|_{X_1}\) and \(V|_{X_2}\) are semistable.
Then \(V\) is semistable.
Furthermore, if one of \(V|_{X_1}\) or \(V|_{X_2}\) is stable, then \(V\) is stable.
\end{lem}

\section{Rational curves}\label{sec-rational}
In this section, we will survey results about the restricted tangent bundle and the normal bundle of rational curves in projective space.

\subsection{Bundles on rational curves} A vector bundle $V$ on $\PP^1$ decomposes as a direct sum of line bundles, that is, there is a unique sequence of integers \[a_1 \leq a_2 \leq \cdots \leq a_r\] such that \[V \cong \bigoplus_{i=1}^r \OO_{\PP^1}(a_i).\] Hence, there are no stable bundles of higher rank on $\PP^1$, and the only semistable bundles are of the form $\OO_{\PP^1}(a)^{\oplus r}$ and exist only when the rank divides the degree $d=ra$. The following notion is the closest possible alternative to (semi)stability for vector bundles on $\PP^1$.

\begin{defin}
  The vector bundle $V$ on $\PP^1$ is called \defi{balanced} when $a_r - a_1 \leq 1$ and  \defi{perfectly balanced} when $a_1 = a_r$. The bundle is perfectly balanced if and only if it is semistable. More generally, the bundle is called \defi{$j$-balanced} if $a_r - a_1 \leq j$.  
\end{defin}

The bundle $V$ is $j$-balanced if and only if every summand of the endomorphism bundle $\End(V) \cong V \otimes V^{\vee}$ has degree at least $-j$. Hence, $j$-balanced bundles can be characterized by the vanishing of $H^1(\End(V)(j-1))$. Consequently, the property of being $j$-balanced is an open condition in flat families. 

In a flat family of bundles on $\PP^1$, the expected codimension of a given splitting type $V \cong \bigoplus \OO_{\PP^1}(a_i)$ is given by $$h^1(\End(V)) = \sum_{\{ i,j| a_i - a_j \leq -2\}} (a_j - a_i -1).$$ This codimension agrees with the codimension of scrolls of type $S_{a_1, \dots, a_r}$ in the Hilbert scheme of scrolls in projective space (see, for example, \cite[Lemma 2.4]{coskungw}) and is the codimension of the given splitting type in the versal deformation space of bundles on $\PP^1$ with that given splitting type.

\subsection{The restricted tangent bundle of rational curves in projective space} A rational curve of degree $e$ can be viewed as the image of a morphism $f\colon \PP^1 \to \PP^r$, where $f$ is given by $r+1$ homogeneous polynomials of degree $e$ in two variables. Pulling back the Euler sequence to $\PP^1$ via $f$, we obtain the exact sequence 
\[0 \to \OO_{\PP^1} \stackrel{(f_0, \dots, f_r)^T}{\longrightarrow} \OO_{\PP^1}(e)^{\oplus(r+1)} \to T\PP^r|_X \to 0.\] Consequently, the splitting type of $T\PP^r|_X$ is determined by the degrees of the generators of the module of relations among the defining polynomials of $f$.

\begin{example}\label{ex-p1restanbun}
    In the case of the rational normal curve
    \[ (s^r: s^{r-1}t:\cdots : st^{r-1} : t^r)\]
    the relations among the columns are given by \[ t \cdot s^i t^{r-i} = s \cdot s^{i-1} t^{r-i+1}.\] Hence, the splitting type of $T\PP^r|_X$ when $X$ is a rational normal curve is $\OO_{\PP^1}(r+1)^{\oplus r}$.
\end{example}

\noindent
An important consequence of Example \ref{ex-p1restanbun} is the following corollary.

\begin{cor}
    The tangent bundle $T\PP^r$ is a stable bundle.
\end{cor}
   \begin{proof}
       Let $H$ be the class of a hyperplane on $\PP^r$. Since the rank $r$ and the degree $r+1$ of $T \PP^r$ are relatively prime, it suffices to prove that $T \PP^r$ is semistable. Let $W$ be a  subsheaf of $T\PP^r$. If necessary, by taking the double dual, we may assume that the singularities of $W$ have codimension at least 3. Consequently, a general rational normal curve $X$ does not intersect the singularities of $W$ and the restriction of $W$ to $X$ is a subsheaf of $T\PP^r|_X$. Since the latter bundle is semistable, \[r \mu_H(W) = \mu(W|_X) \leq \mu(T\PP^r|_X) = r\mu_H(T\PP^r). \qedhere\]
   \end{proof} 
    
This corollary demonstrates that one can sometimes prove the stability of important bundles by proving the stability of their restriction to a curve. 

\begin{example}\label{ex-monomialcurve}
    Let $k_0=0< k_1= 1 < k_2 < \cdots < k_{r-1}=e-1 < k_r = e$ be an increasing sequence of nonnegative integers. Consider the monomial rational curve $X$ given by
    \[(s^e: s^{e-1} t^: \cdots  : s^{e-k_i}t^{k_i}: \cdots : st^{e-1}: t^e).\]
    For $1 \leq i \leq r$, let $b_i = k_i - k_{i-1}$. Then the relations among the columns are generated by \[ t^{b_i} \cdot s^{e-k_{i-1}} t^{k_{i-1}} = s^{b_i} \cdot s^{e-k_i} t^{k_i}. \] We conclude that 
    \[T\PP^r|_X \cong \bigoplus_{i=1}^r \OO_{\PP^1}(e+b_i).\]
\end{example}

Recall that a rational curve $f\colon \PP^1 \to \PP^r$ is \defi{unramified} or \defi{immersed} if the natural map \[f^* \Omega \PP^r \to \Omega \PP^1\] is surjective. In the classical literature, some authors refer to the image of such a map as a rational curve with \defi{ordinary singularities}.

Given a flat family of vector bundles $\mathcal{V} \to S$ parameterized by a base $S$, the Kodaira--Spencer map at $s \in S$ is the map from $T_s S \to \Ext^1(\mathcal{V}_s, \mathcal{V}_s)$ that associates to a Zariski tangent vector to $S$ at $s$ the corresponding infinitesimal deformation of the bundle $\mathcal{V}_s$. For rational curves with ordinary singularities, we claim that the Kodaira--Spencer map for the restricted tangent bundle is surjective. We can parameterize such rational curves as an open subset in  $\Hom(\OO_{\PP^1}, \OO_{\PP^1}(e)^{\oplus (r+1)})$ corresponding to the fiberwise injective morphisms. The Kodaira--Spencer map 
\[\kappa \colon \Hom\left(\OO_{\PP^1}, \OO_{\PP^1}(e)^{\oplus (r+1)}\right) \to \Ext^1\left(f^* T\PP^r, f^* T\PP^r\right) \]
factors through the natural morphisms 
\[\Hom\left(\OO_{\PP^1}, \OO_{\PP^1}(e)^{\oplus (r+1)}\right) \stackrel{\phi}{\longrightarrow} \Ext^1\left(f^* T\PP^r, \OO_{\PP^1}(e)^{\oplus (r+1)}\right) \stackrel{\psi}{\longrightarrow}\Ext^1\left(f^* T\PP^r, f^*T\PP^r\right), \]
where $\phi$ and $\psi$ are the maps in the long exact sequences obtained by applying $\Hom(-, \OO_{\PP^1}(e)^{\oplus (r+1)})$ and $\Hom(f^* T\PP^r, -)$ to the pullback of the Euler sequence, respectively.  Since \[\Ext^1 (\OO_{\PP^1}(e)^{\oplus (r+1)}, \OO_{\PP^1}(e)^{\oplus (r+1)}) =0 \quad \mbox{and} \quad\Ext^2(f^* T\PP^r, \OO_{\PP^1})=0,\] we conclude that both $\phi$ and $\psi$ are surjective. Therefore, the Kodaira--Spencer map is surjective. In particular, the restricted tangent bundle on a general rational curve is balanced and the locus of unramified maps $f$ such that $f^* T\PP^r$ has a specified splitting type has the expected codimension. More precisely, Ramella proves the following.

\begin{thm}\cite{ramella}
The locally closed locus in $\Mor_e(\PP^1, \PP^r)$ parameterizing unramified morphisms, where $f^* T \PP^r$ has a specified splitting type, is irreducible of the expected dimension \[(e+1)(r+1) - 1 - h^1(\End(f^* T \PP^r)).\]
\end{thm}

\subsection{Normal bundles of rational curves in projective space} When the characteristic does not divide the degree $e$ of the rational curve $X$, we can use the Euler sequence to compute the normal bundle of a rational curve. We have the following commutative diagram

 \begin{equation}
     \begin{tikzcd}
     & 0 \arrow[d] & 0 \arrow[d] & & \\
          & \OO_{\PP^1} \arrow[r, "e"] \arrow[d,  "(s \ t)^T"]  & \OO_{\PP^1} \arrow[d] & & \\
         0\arrow[r] & \OO_{\PP^1}(1)^{\oplus 2} \arrow[r, "J"]\arrow[d] & \OO_{\PP^1}(e)^{\oplus (r+1)} \arrow[r]\arrow[d] & N_{X/\PP^r} \arrow[r]\arrow[d] & 0 \\
        0 \arrow[r]& T\PP^1 \arrow[r]\arrow[d] & T\PP^r|X \arrow[r]\arrow[d] & N_{X/\PP^r} \arrow[r] & 0 \\
        & 0 & 0 & &
     \end{tikzcd}
 \end{equation}
where the first column is the Euler sequence for $\PP^1$, the second column is the pullback of the Euler sequence on $\PP^r$, the top horizontal map is multiplication by $e$ and $J$ is the Jacobian matrix
\[J=\left(\begin{array}{cccc}
\frac{\partial f_0}{\partial s} & \frac{\partial f_1}{\partial s} & \cdots & 
\frac{\partial f_r}{\partial s} \\
\frac{\partial f_0}{\partial t} & \frac{\partial f_1}{\partial t} & \cdots & 
\frac{\partial f_r}{\partial t}
\end{array}
\right)^T,\] where $T$ denotes the transpose.
The commutativity of the top square follows from the Euler relation for a homogeneous polynomial $g$ in $n+1$ variables \[ \sum_{i=0}^n x_i \frac{\partial g}{\partial x_i} = \deg(g) g\] Consequently, the normal bundle of $X$ is the cokernel of the map given by the Jacobian matrix $J$
\[0 \to \OO_{\PP^1}(1)^{\oplus 2} \stackrel{J}{\longrightarrow} \OO_{\PP^1}(e)^{\oplus (r+1)} \to N_{X/\PP^r} \to 0.\] Hence, the splitting type of $N_{X/\PP^r}$ can be determined by computing the degrees of the generators of the module of relations among the columns of $J^T$.

\begin{rem}
 Observe that this argument fails when the degree $e$ is divisible by the characteristic since multiplication by $e$ is no longer an isomorphism.   
\end{rem}

\begin{example}
   For a rational normal curve $(s^r: s^{r-1}t : \dots : t^r)$, the Jacobian matrix is given by
   \[\left (\begin{array}{ccccccc}
      r s^{r-1}   & (r-1)s^{r-2}t  & \cdots & i s^{i-1}t^{r-i} & \cdots & t^{r-1} & 0 \\
      0   &  s^{r-1} & \cdots & (r-i)s^i t^{r-i-1} & \cdots &(r-1)st^{r-2} &  rt^{r-1}
   \end{array}  \right)^T.\] Let $C_i$ denote the $i$th column of this matrix. Observe that the relations among the columns are generated by the relations
   \[ t^2 C_i - 2st C_{i+1} + s^2 C_{i+2}=0.\]
   It follows that \[N_{X/\PP^r} \cong \OO_{\PP^1}(r+2)^{\oplus (r-1)}.\] Although this calculation assumes that the characteristic does not divide $r$, the answer remains the same even when the characteristic divides $r$ (see, for example, \cite[Corollary 2.5]{coskunriedlci}).
\end{example}

\begin{example}\label{ex:monomial_normal}
    More generally, assume that the characteristic of the base field is greater than $e$. Let $X$ be the monomial rational curve $(s^e: s^{e-1}t: \cdots : s^{k_i} t^{e-k_i}: \cdots : st^{e-1}:t^e)$ introduced in Example \ref{ex-monomialcurve}. Then the Jacobian matrix $J$ has the form 
   \[\left (\begin{array}{ccccccc}
      e s^{e-1}   & (e-1)s^{e-2}t  & \cdots & k_i s^{k_i-1}t^{e-k_i} & \cdots & t^{r-1} & 0 \\
      0   &  s^{e-1} & \cdots & (e-k_i)s^{k_i} t^{e-k_i-1} & \cdots &(e-1)st^{e-2} &  et^{e-1}
   \end{array}  \right)^T.\]  
Three consecutive columns $C_{i-1}, C_i, C_{i+1}$ of the matrix satisfy the relation
\[(k_i - k_{i+1}) s^{k_{i+1}-k_{i-1}}C_{i-1} + (k_{i+1}-k_{i-1}) s^{k_{i+1}-k_i} t^{k_i - k_{i-1}} C_i + (k_{i-1}-k_i) t^{k_{i+1}-k_{i-1}} C_{i+1} = 0.\] These relations generate the module of relations. For $1 \leq i \leq r-1$, set \[c_i:= k_{i+1} - k_{i-1}.\]  We conclude that 
\[ N_{X/\PP^r} \cong \bigoplus_{i=1}^{r-1} \OO_{\PP^1}(e+c_i).\]
\end{example}

\begin{rem}
    The splitting type of the normal bundle of a curve may depend on the characteristic. In small characteristics, there may be lower degree relations among the columns due to some coefficients becoming 0. For example, the normal bundle of the curve \[(s^5: s^4t: s^2t^3: st^4: t^5)\] has balanced splitting type $\OO_{\PP^1}(7) \oplus \OO_{\PP^1}(8)^{\oplus 2}$ in characteristics not equal to 2 or 5. In characteristic~2, the splitting type becomes $\OO_{\PP^1}(7) ^{\oplus 2}\oplus \OO_{\PP^1}(9)$ (see \cite[Proposition 3.7]{coskunriedlci}).
\end{rem}

When the characteristic of the base field is 0, Sacchiero \cite{sacchiero80} determined the possible splitting types of normal bundles of rational curves with ordinary singularities in $\PP^r$.

\begin{rem}
    Suppose that a rational curve $X$ of degree $e$ is degenerate, i.e., that it spans a $\PP^s \subset \PP^r$ for \(s < r\). Then \[N_{X/\PP^r} \cong N_{X/\PP^s} \oplus \OO_{\PP^1}(e)^{\oplus (r-s)}.\] To see this, observe that we have the normal bundle exact sequence
    \[ 0 \to N_{X/\PP^s} \to N_{X/\PP^r} \to N_{\PP^s/\PP^r}|_X \to 0.\] Since $N_{\PP^s/\PP^r}\cong \OO_{\PP^s}(1)^{\oplus (r-s)}$, we conclude that 
    \[N_{\PP^s/\PP^r}|_X \cong \OO_{\PP^1}(e)^{\oplus (r-s)}.\] Since each summand in $N_{X/\PP^s}$ has degree at least $e$, we conclude that the sequence splits. Consequently, in order to understand normal bundles of rational curves, one can restrict to nondegenerate rational curves, which we will do for the rest of the discussion.   
\end{rem}

It is not always possible to find a monomial rational curve as in Example~\ref{ex:monomial_normal} with balanced normal bundle. For example, such a rational curve in \(\PP^4\) has the form \((s^e : s^{e-1}t : s^kt^{e - k}: st^{e-1}: t^e)\),
so has normal bundle \(\OO_{\PP^1}(e - k) \oplus \OO_{\PP^1}(e - 2) \oplus \OO_{\PP^1}(k)\), which is never balanced for \(e \geq 7\). We must therefore consider a more general class of rational curves.

Let \(2 \leq b_1 \leq \cdots  \leq b_{r - 1}\) be integers such that \(\sum_{i = 1}^{r - 1} b_i = 2e - 2\). Define $\delta_1 = 1$, and $\delta_i = b_{i-1} - \delta_{i-1}$ for $1 < i \leq r-1$. Let $c= 1 + \sum_{i=1}^{r-1} \delta_i$. Let $p(s,t)$ and $q(s,t)$ be two general homogeneous polynomials of degree $e-c$. Let $k_i = c-\sum_{j=1}^i \delta_j$ for $0 \leq i \leq r-1$. Observe that $k_0=c$, $k_1= c-1$ and $k_{r-1} = 1$. Consider the rational curve $f\colon \PP^1 \to \PP^r$ defined by 
\[f= (s^{k_0}p: s^{k_1}t^{c-k_1}p: \dots : s^{k_{r-1}}t^{c-k_{r-1}}p: t^cq).\]
\begin{lem}[Sacchiero's Lemma]
The map $f$ is unramified and the normal bundle has the splitting type 
\[N_f \cong \bigoplus_{i=1}^{r-1} \OO_{\PP^1}(e+b_i).\]
\end{lem}

\begin{cor}
    Assume that the characteristic of the base field is larger than $e$. For $1 \leq i \leq r-1$, let  $b_i \geq 2$ be integers such $\sum_{i=1}^{r-1} b_i = 2e-2$. Then there exists an unramified, nondegenerate map $f\colon \PP^1 \to \PP^r$ of degree $e$ such that the normal bundle $N_f$ has the splitting type
    \[ N_f \cong \bigoplus_{i=1}^{r-1} \OO_{\PP^1}(e+b_i).\]
\end{cor}

\noindent
In any characteristic, the normal bundle can be computed via the exact sequence
\[0 \to N_{X/\PP^r}^{\vee} (1) \to \OO_{\PP^1}^{\oplus(r+1)} \to \mathcal{P}(\OO_{\PP^1}(e))\to 0,\]
where $\mathcal{P}(\OO_{\PP^1}(e))$ is the principal parts bundle associated to $\OO_{\PP^1}(e)$; see \cite[Section 2.4]{clvgrassmannian}. Let $F\colon X \to X$ denote the relative Frobenius morphism. When the ground field has characteristic 2, the principal parts bundle can be expressed as \[\mathcal{P}(\OO_{\PP^1}(e)) = F^* F_* \OO_{\PP^1}(e).\] Consequently, the last two bundles are Frobenius pullbacks of bundles on $\PP^1$, and with not too much work you can see that the map between them is also pulled back under \(F\).  We conclude that $N_{X/\PP^r}^{\vee}(1)$ is also the pullback of a bundle via the Frobenius. Hence, all its summands must be divisible by 2. We conclude the following proposition.

\begin{prop}
    Assume that the ground field has characteristic 2 and let $X$ be a rational curve of degree $e$ in $\PP^r$. Then $N_{X/\PP^r} \cong \bigoplus_{i=1}^{r-1} \OO_{\PP^1}(a_i),$ where all $a_i \equiv e \mod 2$.
\end{prop}

In particular, in characteristic 2, there are further restrictions on the splitting type of the normal bundle of a rational curve, and the normal bundle of a general nondegenerate rational curve is not necessarily balanced. The following theorem summarizes the splitting type of the normal bundle of a general rational curve (see \cite{clvgrassmannian, coskunriedl, interpolation, sacchiero, ran}). 

\begin{thm}
The normal bundle of a general rational curve in $\PP^r$ is 2-balanced. 
    Moreover, if the base field does not have characteristic 2, then the general nondegenerate rational curve has balanced normal bundle. If the characteristic of the base field is 2, then the normal bundle of the general rational curve is the unique bundle $\bigoplus_{i=1}^{r-1} \OO_{\PP^1}(a_i)$, which is both 2-balanced and satisfies $a_i \equiv e \mod 2$.
\end{thm}

We now turn to the problem of understanding the loci of rational curves whose normal bundle has a specified splitting type. In $\PP^3$, Eisenbud and Van de Ven show that these loci are irreducible of the expected dimension (see \cite{eisenbudvandeven, ghionesacchiero}).

\begin{thm} \cite{eisenbudvandeven} 
    In characteristic 0, the locus of smooth nondegenerate rational curves $X$ of degree $e$ in $\PP^3$ whose normal bundle has a specified splitting type is irreducible of the expected dimension.  
\end{thm}

Eisenbud and Van de Ven conjectured that similar results should hold in $\PP^r$. The first counterexample was given by Alzati and Re \cite{AlzatiRe} who showed that the locus of rational curves of degree 11 in $\PP^8$ whose normal bundle has splitting type $\OO_{\PP^1}(13)^{\oplus 3} \oplus \OO_{\PP^1}(14)^{\oplus 2} \oplus \OO_{\PP^1}(15)^{\oplus 2}$ has two components.  Coskun and Riedl showed that when $r \geq 5$, the loci of rational curves whose normal bundle has a given splitting type in general have many irreducible components of different dimensions \cite{coskunriedl}.

\begin{problem}
    Describe the components of the locus of rational curves in $\PP^r$ whose normal bundles have a specified splitting type. Give sharp upper and lower bounds on the dimensions of these loci.
\end{problem}

Since rational curves play a central role in the study of Fano varieties, it is interesting to study the restricted tangent and normal bundles of rational curves on other Fano varieties. We will next survey some recent progress in this direction.

\subsection{Restricted tangent bundles of rational curves in Grassmannaians}

Let $G(k,n)$ denote the Grassmannian parameterizing $k$-dimensional subspaces of an $n$-dimensional vector space $V$. The Grassmannian comes equipped with two tautological bundles: 
\begin{enumerate}
    \item The rank $k$ tautological subbundle $S$ associates to a point $[W] \in G(k,n)$ the $k$-dimensional subspace $W$.
    \item The rank $n-k$ tautological quotient bundle $Q$ associated to a point $[W]\in G(k,n)$ the $(n-k)$-dimensional quotient $V/W$.
\end{enumerate}
The tangent bundle of the Grassmannian has a tensor structure \[ TG(k,n) \cong S^{\vee} \otimes Q.\] Consequently, the restricted tangent bundle $TG(k,n)|_X$ cannot be balanced unless the degree $e$ of the rational curve is divisible by $k$ or $n-k$.

\begin{example}
    Let $X$ be a general twisted cubic curve in $G(2,4)$. Then the restriction of  $S^{\vee}$ and $Q$ to $X$ both split as $\OO_{\PP^1}(1) \oplus \OO_{\PP^1}(2)$. Hence, $TG(2,4)|_X \cong \OO_{\PP^1}(2) \oplus \OO_{\PP^1}(3)^{\oplus 2} \oplus \OO_{\PP^1}(4)$ is not balanced. In fact, this phenomenon holds for all odd degree rational curves. 
\end{example}

 Mandal \cite{Ma19} has studied the splitting types of the restricted tangent bundle of rational curves in Grassmannians. He begins by studying the splitting types of $S$ and $Q$. Let $a_1, \dots, a_k$ be nonpositive integers whose sum is $-e$. Let $b_1, \dots, b_{n-k}$ be nonnegative integers whose sum is $e$. Let $S(a_1, \dots, a_k)$ be the locus of rational curves $X$ of degree $e$ in $G(k,n)$ for which $S|_X \cong \bigoplus_{i=1}^k \OO_{\PP^1}(a_i)$. Similarly, let $Q(b_1, \dots, b_{n-k})$  be the locus of rational curves $X$ of degree $e$ in $G(k,n)$ for which $Q|_X \cong \bigoplus_{j=1}^{n-k} \OO_{\PP^1}(b_j)$. First, Mandal proves the following proposition.

\begin{prop} \cite[Proposition 5]{Ma19} 
   The loci  $S(a_1, \dots, a_k)$ and $Q(b_1, \dots, b_{n-k})$ are smooth and of the expected dimension. In particular, for a general rational curve in $G(k,n)$, the restriction of $S$ and $Q$ are both balanced.
\end{prop}

\noindent
For a curve $X$ in the intersection of $S(a_1, \dots, a_k)$ and $Q(b_1, \dots, b_{n-k})$, we have \[TG(k,n)|_X \cong \bigoplus_{i,j} \OO_{\PP^1}(b_j -a_i).\]
In characteristic 0, Mandal further proves the following theorem.

\begin{thm} \cite[Theorem 22]{Ma19}
    The intersection of $S(a_1, \dots, a_k)$ and $Q(b_1, \dots, b_{n-k})$ is nonempty and generically transverse. 
\end{thm}
In particular, this theorem determines all the possible splitting types of $TG(k,n)|_X$. However, unlike the case of $\PP^r$, the locus of rational curves where $TG(k,n)|_X$ has a specified splitting type is not necessarily irreducible. Different splitting types of $S$ and $Q$ may result in the same splitting type for $TG(k,n)|_X$. This often produces many different irreducible components of the locus where $TG(k,n)|_X$ has a given splitting type. Mandal proves lower and upper bounds for the number of components and shows that there is always a component of the expected dimension.

\subsection{Normal bundles of rational curves in Grassmannians}
The fact that $TG(k,n)|_X$ is not balanced under certain numerical conditions obstructs the normal bundle  $N_{X/G(k,n)}$ from being balanced. 
By the division algorithm, write 
$$e= kq_1 + r_1 \ \ \mbox{and} \ \ e= (n-k)q_2 + r_2$$ with $0 \leq r_1 < k$ and $0 \leq r_2 < n-k$. If $X$  is a general rational curve of degree $e > 1$ in $G(k,n)$, then the splitting type of $TG(k,n)|_X$ is
\[\OO_{\PP^1}(q_1 + q_2 + 2)^{\oplus r_1 r_2} \oplus \OO_{\PP^1} (q_1 + q_2 +1)^{\oplus (r_1 (n-k-r_2)+ r_2(k-r_1))} \oplus \OO_{\PP^1}(q_1+q_2)^{\oplus(k-r_1)(n-k-r_2)}.\]
If $q_1+ q_2-1 < (k-r_1)(n-k-r_2)$, then twisting the sequence 
\[   0 \to T_X  \to TG(k,n)|_X \to N_{X/G(k,n)} \to 0\] 
 by $\OO_{\PP^1}(-q_1-q_2-2)$ and comparing the first cohomology groups, we see that $N_{X/G(k,n)}$ must have a summand equal to  $\OO_{\PP^1}(q_1+q_2)$. On the other hand, if $r_1 r_2 \not= 0$, then $N_{X/G(k,n)}$ must also have a summand equal to $\OO_{\PP^1}(t)$ with $t \geq q_1 + q_2 + 2$. Hence, the normal bundle cannot be balanced. However, we have the following theorem.

\begin{thm}\cite[Theorem 1.1]{clvgrassmannian}
    Let $X$ be a general rational curve of degree $e$ in $G(k,n)$. Then the normal bundle $N_{X/G(k,n)}$ is $2$-balanced.
\end{thm}

In analogy with $\PP^r$, it is natural to expect that the normal bundle of a general rational curve $N_{X/G(k,n)}$ is balanced if the degree $e$ of the curve is sufficiently large (and the characteristic is different from 2). In fact, in \cite{clvgrassmannian} the authors made the following more precise conjecture.

\begin{conj}\cite[Conjecture 1.2]{clvgrassmannian}
    Let $X$ be a general rational curve of degree $e$ in $G(k, n)$, where without loss of generality we take \(k \leq n-k\).
Then $N_{X/G(k, n)}$ is balanced if and only if either \(e = 1\) or none of the following occur:
\begin{itemize}
\item The degeneracy exceptions: We have \(e < n-k\) and \(e \neq k\).
\item The characteristic \(2\) exceptions: The characteristic is $2$ and \(k = 1\) and \(e \not\equiv 1\) mod \(n-k - 1\).
\item The tangent bundle splitting exceptions: We have $q_1+ q_2 \leq (k-r_1)(n-k-r_2)$ and \(r_1 r_2 \neq 0\).
\end{itemize}
\end{conj}

Recently, An Cao has resolved the conjecture in cases \((k, n) = (2, 4)\), \((2, 5)\), and \((2, 6)\) when the characteristic is \(0\) \cite{cao}.

\subsection{Restricted tangent bundles of rational curves in hypersurfaces}
 In characteristic 0, Mioranci has studied restricted tangent bundles of rational curves in general hypersurfaces in $\PP^r$ \cite{mioranci2}. His main theorem is the following.

\begin{thm}\cite[Theorem 1.1]{mioranci2}
  Let $Y \subset \PP^r$  be a smooth Fano 
hypersurface of degree $d$, with $3 \leq d \leq r$.
\begin{enumerate}
    \item The restricted tangent bundle $TY|_X$ of a degree $e \leq \frac{r-1}{r+1-d}$ rational curve $X \subset Y$ is never balanced.
    \item A general hypersurface $Y$ contains rational curves of degree $e$ with balanced restricted
tangent bundle for every $e > \frac{r-1}{r+1-d}$.
\end{enumerate}
\end{thm}

Under the Pl\"{u}cker embedding, the Grassmannian $G(2,4)$ is a smooth quadric hypersurface in $\PP^5$. We have already seen that for odd degree rational curves $X$ on $G(2,4)$, the restricted tangent bundle $TG(2,4)|_X$ is not balanced. Mioranci shows that this phenomenon continues to hold for odd degree rational curves on quadric hypersurfaces in $\PP^r$.

\begin{thm}\cite[Theorem 1.2]{mioranci2}
Let $Y \subset \PP^r$, $r\geq 3$, be a quadric hypersurface.
\begin{enumerate}
    \item For every even degree $e \geq 2$, the quadric $Y$ contains a degree $e$ rational curve $X$ with balanced restricted tangent bundle $TY|_X \cong \OO_{\PP^1}(e)^{r-1}$.
    \item The restricted tangent bundle $TY|_X$ is never balanced for an odd degree rational curve \(X\). For every odd $e \geq 1$, there exists a degree $e$ rational curve $X$ whose restricted tangent bundle satisfies $TY|_X \cong \OO_{\PP^1}(e-1) \oplus \OO_{\PP^1}(e)^{\oplus(r-3)} \oplus \OO_{\PP^1}(e+1)$ (so is in particular \(2\)-balanced).
\end{enumerate}
\end{thm}

It would be interesting to study the analogous question on Fano complete intersections in $\PP^r$ and more generally on Fano complete intersections in homogeneous varieties. 

\begin{problem}\label{prob-restan}
    Let $Y$ be a Fano manifold. Determine the possible splitting types of $TY|_X$ where $X$ is a rational curve. Describe the irreducible components of the locus of rational curves where $TY|_X$ has a given splitting type and determine their dimensions. 
\end{problem}

\subsection{Normal bundles of rational curves in complete intersections in homogeneous varieties} We begin with the case of $\PP^r$.
Let $X$ be a smooth rational curve of degree \(e\) in \(\PP^r\), and consider the natural map
\[I_{X / \PP^r} \to \mathcal{H}\mbox{om}(N_{X/\PP^r}, \OO_X).\]
Twisting by \(\OO_{\PP^r}(d)\) and taking global sections, we obtain a natural map \[\phi\colon H^0(I_{X/\PP^r}(d)) \to \Hom(N_{X/\PP^r}, \OO_{\PP^1}(ed)).\]
The main observation is that this map is surjective if $d\geq 3$ and $X$ is a rational curve of degree $e \leq r$ that spans a $\PP^e$ (see \cite[\S 3]{coskunriedlci}). As a consequence, one obtains the following theorem.

\begin{thm}\cite[Main Theorem]{coskunriedlci}\label{thm-balancedci}
    Let $2 \leq d_1 \leq d_2 \leq \cdots \leq d_c$ be positive integers such that $\sum_{i=1} d_i \leq r$. Let $Y \subset \PP^r$ be a general Fano complete intersection of type $(d_1, \dots, d_c)$. 
    \begin{enumerate}
        \item If $d_c \geq 3$, then $Y$ contains rational curves with balanced normal bundle of every degree $1 \leq e \leq r$.
        \item If $d_1 = \cdots = d_c=2$ and $r \geq 2c+1$, then $Y$ contains rational curves with balanced normal bundle of every degree $1 \leq e \leq r-1$.
    \end{enumerate}
\end{thm}

\noindent
Ran has also studied rational curves on Fano hypersurfaces (see \cite{ran21, ran23}).

An irreducible variety $Y$ is \defi{separably rationally connected ({SRC})}  if there exists a variety $Z$ and a morphism $e\colon Z\times \PP^1\to Y$ such that the induced morphism on products,
\[
e^{(2)}\colon Z\times \PP^1\times \PP^1 \to Y\times Y,
\]
is smooth (and hence dominant). We refer the reader to \cite{Kol96} for a  discussion of the properties of SRC varieties. A rational curve $X \subset Y$ is called \defi{free} if the normal bundle $N_{X/Y}$ is globally generated, or equivalently, if all the summands in the splitting type of $N_{X/Y}$ have nonnegative degree. 
 The curve $X$ is called \defi{very free} if $N_{X/Y}$ is ample, or equivalently, if all the summands in the splitting type of $N_{X/Y}$ have positive degree.
If $Y$ is a smooth variety over an algebraically closed field, then $Y$ is SRC if $Y$ contains a  very free rational curve \cite[Theorem IV.3.7]{Kol96}.

In characteristic 0, rational connectedness and separable rational connectedness agree and smooth Fano varieties are SRC \cite[Theorem V.2.13]{Kol96}.  Koll\'{a}r points out that SRC is the suitable generalization of rational connectedness to arbitrary characteristic and poses the question: Is every smooth Fano variety in positive characteristic SRC? Koll\'{a}r's question has been answered affirmatively for general Fano complete intersections in $\PP^r$  \cite{CZ14,Tia15}. Likewise, if $Y$ is any smooth Fano complete intersection of type $(d_1,\ldots,d_c)$ of Fano index at least 2 and the characteristic of the base field is greater than $\max(d_1,\ldots,d_c)$, then $Y$ is SRC \cite{STZ18}. Theorem \ref{thm-balancedci} has the following important consequence (see \cite[Theorem 5.6]{coskunriedlci}).

\begin{thm}
    Let $Y \subset \PP^r$ be a general Fano complete intersection of type $(d_1, \dots, d_c)$. Set $d = \sum_{i=1}^c d_i$. Then $Y$ contains very free rational curves of every degree $e \geq \left\lceil \frac{r-c+1}{r-d+1} \right\rceil$. In particular, $Y$ is SRC.
\end{thm}

\begin{rem}
    Certain special Fano hypersurfaces are known not to have  very free curves of low degree \cite{She12, Bridges, cheng}. 
\end{rem}

This discussion can be generalized to Fano complete intersections in other homogeneous varieties or weighted projective spaces (see \cite{coskunsmithvf}). The main technical tool is the following theorem.

\begin{thm}\cite[Theorem 2.1]{coskunsmithvf}\label{csmain}
Let $Y \subset \PP^r$ be a projective variety whose ideal sheaf is generated in degree $k$. Let $X$ be a rational normal curve of degree $e$ contained in the smooth locus of $Y$. Let $1 \leq c\leq \dim(Y)-2$ be an integer. For $1 \leq i \leq c$, let $D_i = d_i H + E_i$ be Cartier divisor classes on $Y$, where $d_i \geq \max(k,3)$, $H$ is the hyperplane class and $E_i$ are effective divisors such that the restriction map
\[
H^0(Y,E_i)\to H^0(X,E_i|_X)
\]
is surjective. Given a surjective map $$q\in \mathrm{Hom}\left(N_{X/Y}, \bigoplus_{1\leq i \leq c} \OO(D_i)|_X\right),$$ there are hypersurfaces $Z_i$ with $[Z_i]=D_i$ such that if $Z=\bigcap_{i=1}^c Z_i$, then $Z$ is smooth of codimension $c$ along $C$ and $N_{X/Z} \cong \ker q$. 
\end{thm}

\noindent
For example, when $Y= G(k,n)$, Theorem \ref{csmain} specializes to the following theorem.

\begin{thm}
  Let $d_1 , \dots ,  d_c \geq 3$ with $c \geq 1$ be integers such that $\sum_{i=1}^c d_i < n$. Let $X$ be a rational normal curve of degree $e$ in $G(k,n)$, i.e., a rational curve of degree \(e\) whose image under the Pl\"ucker embedding spans a \(\PP^e\), and let $Z_i$ be general hypersurfaces of degree $d_i$ in $G(k,n)$ containing $X$. Let $Z = \bigcap_{i=1}^c Z_i$. 
  Then $N_{X/Z}$ is balanced. In particular, $Z$ is SRC.
\end{thm}

We refer the reader to \cite{coskunsmithvf} for analogous statements for products of projective spaces, flag varieties, Schubert varieties, and weighted projective spaces. The following problem remains very much open in general.

\begin{problem}\label{prob-restan}
    Let $Y$ be a Fano manifold. Determine the possible splitting types of $N_{X/Y}$ where $X$ is a rational curve in $Y$. Describe the irreducible components of the locus of rational curves where $N_{X/Y}$ has a given splitting type and determine their dimensions. 
\end{problem}

 Some progress has been made on describing the locus of rational curves whose normal bundles have a specified splitting type when $Y$ is a general hypersurface in $\PP^r$ and $X$ is a rational curve of small degree. For example, H.\ Larson \cite{hannah} shows that the locus of lines $L$ in a general hypersurface $Y$ in $\PP^r$ where $N_{L/Y}$ has a specified splitting type has the expected dimension. H.\ Larson further computes the class of this locus in the Grassmannian $G(2, r+1)$. Similarly, Mioranci studies the locus of rational normal curves $X$ on general hypersurfaces $Y$ in $\PP^r$ where $N_{X/Y}$ has a specified splitting type \cite{mioranci}.

\section{Genus 1 curves}\label{sec-elliptic}
In this section, we discuss the restricted tangent bundle and the normal bundle of genus 1 curves in projective space.

\subsection{Semistable bundles on genus 1 curves} Vector bundles on genus 1 curves have been classified by Atiyah \cite{Atiyah}. Let $X$ be a curve of genus 1. There exists a stable bundle of rank $r$ and degree $d$ on $X$ if and only if $d$ and $r$  are relatively prime. In this case, such stable bundles form a torsor under the Jacobian of \(X\). If the $\gcd(r,d)=a$, write $r=ar_1$ and $d=ad_1$. Then every indecomposable bundle of rank \(r\) and degree \(d\) can be expressed as an iterated extension of a stable bundle of rank $r_1$ and degree $d_1$ by itself (and is, in particular, semistable).  Moreover, any vector bundle on \(X\) can be written as a direct sum of indecomposible bundles, and the indecomposible bundles appearing in that direct sum decomposition are unique.  

\begin{example}
    Let $L$ be a line bundle on  $X$ and let $p\in X$ be a point. Then $\Ext^1(L,L)$ and $\Ext^1(L(p), L)$ are both one-dimensional. Consequently, there exist (up to scaling) unique nonsplit extensions
\[0\to L \to V_{L} \to L \to 0 \quad \mbox{and} \quad 0 \to L \to V_{L,p} \to L(p) \to 0.\] The bundle $V_{L}$ is indecomposible, and the bundle $V_{L,p}$ is stable.
\end{example}

Given any automorphism $\phi\colon X \to X$, we get an induced action $\phi_a^*\colon \Pic^a(X) \to \Pic^a(X)$ via pullback.  A map $f\colon \Pic^a(X) \to \Pic^b(X)$ is \defi{natural} if for every automorphism $\phi\colon X \to X$, the map $f$ commutes with the induced actions: \[\phi_b^* \circ f = f \circ \phi_a^*\] Let $\phi\colon X \to X$ be translation by a point of order $a$. Since $\phi_a^*$ is the identity on $\Pic^a(X)$, if $f\colon \Pic^a(X) \to \Pic^b(X)$ is a natural map, then $\phi_b^*$ must also be the identity. We conclude the following important observation.
\begin{lem}\cite[Lemma 2.7]{clvcanonical}\label{lem-blackmagic}
    If $f \colon \Pic^a(X) \to \Pic^b(X)$ is a natural map, then $a$ divides $b$.
\end{lem}

Fix a genus 1 curve $X$, and consider the embeddings of $X$ into $\PP^r$ of degree $d$. The first Chern classes of the successive quotients of the HN-filtrations of $T\PP^r|_X$ and $N_{X/\PP^r}$ are canonically associated to the embedding. Since the space of such embeddings with fixed \(\OO_X(1)\in \Pic^d(X)\) forms a rational variety, and every rational map from a rational variety to an abelian variety is constant, these first Chern classes therefore depend naturally on \(\OO_X(1)\in \Pic^d X\).
Therefore, by Lemma \ref{lem-blackmagic}, their degrees have to be divisible by $d$ for the generic such embedding. This imposes strong conditions on possible HN-filtrations of the bundles $T\PP^r|_X$ and $N_{X/\PP^r}$.

\subsection{\boldmath The restricted tangent bundles of genus 1 curves in $\PP^r$} 
We demonstrate how this observation can be used to prove the semistability of $T\PP^r|_X$ for a general nondegenerate genus 1 curve $X$ of degree $d$. Consider the restriction of the Euler sequence to $X$
\[0 \to \OO_X \to \OO_X(1)^{\oplus (r+1)} \to T\PP^r|_X \to 0.\] 
Suppose $T\PP^r|_X$ is not semistable. Let $Q$ be the minimal slope semistable quotient of $T \PP^r|_X$, whose degree must be divisible by $d$. Moreover, since $Q$ is a quotient of $\OO_X(1)^{\oplus (r+1)}$, if the slope of $Q$ is $d$, then $Q$ must be an iterated extension of $\OO_X(1)$ and would admit a surjective map to $\OO_X(1)$. Composing with the map from the Euler sequence, we would conclude that the sections of $\OO_X(1)$ satisfy a linear relation, contrary to the assumption that $X$ is nondegenerate. Therefore, we conclude that 
\[ d  < \mu(Q) < \frac{d(r+1)}{r}.\]
Assume that $Q$ has rank $s$. Then multiplying these inequalities by $s$, we obtain
\[ ds < \deg(Q) < ds + \frac{ds}{r}.\]
Since $d$ divides $\deg(Q)$, we obtain an integer $q$ of the form
\[ s < q < s + \frac{s}{r}.\] Since $s<r$, there cannot be such an integer $q$, leading to a contradiction. We conclude that $T \PP^r|_X$ is semistable. In fact, Farkas and E.\ Larson have shown the following more detailed theorem.  

\begin{thm}\cite[Theorem 1.2]{farkaslarson}
    Let $X \subset \PP^r$ be a general curve of genus \(1\) and degree $e$. Let $a= \gcd(r,e)$ and let $r=ar_1$ and $e= ae_1$. Then \[T\PP^r|_X \cong \bigoplus_{i=1}^a V_i,\] where $V_i$ are stable vector bundles on $X$ of rank $r_1$ and degree $(r+1) a_1$ and $(\det(E_1), \dots, \det(E_a))$ is general in $\prod_{i=1}^a \Pic^{(r+1)d_1}$. 
\end{thm}

\subsection{\boldmath The normal bundles of curves of genus 1 in $\PP^r$}

\begin{example}
    Let $X$ be a nondegenerate curve of degree $4$ and genus 1 in $\PP^3$. Then $X$ is the complete intersection of two quadric surfaces, and $N_{X/\PP^3} \cong \OO_X(2) \oplus \OO_X(2)$. 
\end{example}

\begin{example}
   Let $X$ be a nondegenerate curve of degree $6$ and genus 1 in $\PP^5$. If the characteristic of the base field is not 2, then the embedding line bundle $\OO_X(1)$ has four distinct square roots. Each of these square roots gives an embedding of $X$ in $\PP^2$ as a degree $3$ curve. The curve $X$ is contained in the four Veronese surfaces corresponding to these four square roots. Consequently, $N_{X/\PP^4}$ is a direct sum of four distinct line bundles of degree $9$. In characteristic 2, the line bundle $\OO_X(1)$ has only two square roots and the normal bundle is a direct sum of the rank 2 nonsplit self-extensions of these two bundles. 
\end{example}

Ein and Lazarsfeld \cite[Theorem 4.1]{einlazarsfeld} proved that the normal bundle of an elliptic normal curve (i.e., a nondegenerate curve of degree $r+1$ and genus 1 in $\PP^r$) is semistable. 

\begin{thm}[{\cite[Theorem 3.1]{coskunsmithnb}}]\label{genus1}
Let $X\subset \PP^r$ be a general BN-curve of degree $d\geq r+1 \geq 4$ and genus $g=1$. Then $N_{X/\PP^r}$ is semistable.
\end{thm}

\begin{proof}
Extend the statement to cover \(r = 2\) by replacing the normal bundle with the normal sheaf of the map \(X \to \PP^2\), which is a line bundle and thus automatically semistable. This serves as the base case of our induction on \(r\).
For the inductive step we use the exact sequence of projection from a general point \(p \in X\):
\[0 \to [N_{X \to p} \simeq \OO_X(1)(2p)] \to N_{X/\PP^r} \to N_{\bar{X}/\PP^{r - 1}}(p) \to 0.\]
By our inductive hypothesis the quotient is semistable.
If \(N_{X/\PP^r}\) were not semistable,
the slope of its maximal destabilizing subbundle \(S\) would thus satisfy:
\[d + \frac{2d}{r - 1} = \mu(N_{X/\PP^r}) < \mu(S) \leq \max\left(\mu(N_{X \to p}) = d + 2, \mu(N_{\bar{X}/\PP^{r - 1}}(p)) = d + \frac{2d - 2}{r - 2}\right) < d + \frac{2d}{r - 2}.\]
By Lemma \ref{lem-blackmagic} and the following discussion, \(\deg(S)\) is a multiple of \(d\). Writing \(s\) for the rank of \(S\), we would therefore have that \(\deg(S)/d\) is an integer satisfying
\[s + \frac{2s}{r - 1} < \frac{\deg(S)}{d} < s + \frac{2s}{r - 2}.\]
But there are no such integers for any \(s\) with \(0 < s < r - 1\).
\end{proof}

\noindent
In fact, more precise information can be obtained.

\begin{prop}\label{prop-g1distinct}\cite[Proposition 3.3]{coskunsmithnb}
Let $r \geq 4$ and let $C \subset \PP^r$ be a general, nondegenerate  curve of degree $d$ and genus \(g=1\) where $r-1$ divides $2d$.
\begin{enumerate}
 \item If the characteristic of the base field is $2$, and $r$ is odd, and $r-1$ does not divide $d$, then $N_{C/\PP^r}$ is a direct sum of $\frac{r-1}{2}$ indecomposable rank $2$ bundles. 
    \item Otherwise, $N_{C/\PP^r}$ is a direct sum of distinct line bundles.
   \end{enumerate}
\end{prop}

\section{Higher genus curves}\label{sec-highgenus}
In this section, we survey some recent developments on the stability of restricted tangent bundles and normal bundles of higher genus curves in $\PP^r$. Throughout this section, we assume $g \geq 2$.
Let \[\rho(g,r,d) \colonequals g-(r+1)(g-d+r) \] denote the Brill--Noether number. 
By the Brill--Noether Theorem, a general smooth projective curve $X$ of genus $g$ admits a nondegenerate map of degree $d$ to $\PP^r$ if and only if the Brill-Noether number $\rho(g,r, d)\geq 0$ \cite{griffithsharris, acgh, osserman, jensenpayne}. When $r \geq 3$, the image of a general such map is an embedding \cite{EH83}. Furthermore, there is a unique component of the Hilbert scheme that dominates the moduli space $M_g$ and whose general member parameterizes nondegenerate  genus $g$ curves of degree $d$ in $\PP^r$ \cite{Eishar}. A member of this component is called a \defi{Brill--Noether curve} or \defi{BN-curve} for short.

\begin{rem}
    We warn the reader that there may be other components of the Hilbert scheme parameterizing nondegenerate smooth curves of degree $d$ and genus $g$. However, these components parameterize special curves and do not dominate $M_g$.
\end{rem}

When discussing the stability of the restricted tangent or normal bundle for the ``general curve", it is natural to concentrate on BN-curves. 

\begin{example}
    There are many components of the Hilbert scheme of curves where the general member does not have a (semi)stable normal bundle. For instance, let $X$ be the  complete intersection of hypersurfaces of degrees $1 < d_1 \leq \cdots \leq d_{r-1}$. Then \[N_{X/\PP^r} \cong \bigoplus_{i=1}^{r-1} \OO_X(d_i).\] Hence, $N_{X/\PP^r}$ is semistable if and only if $d_1 = \cdots = d_{r-1}$ and is never stable if $r \geq 3$.
\end{example}

\subsection{Restricted tangent bundles of BN-curves}\label{sec-restrictedtangent} The stability of restricted tangent bundles for higher genus BN-curves was settled by Farkas and E.~Larson \cite{farkaslarson}, generalizing and unifying earlier work including \cite{ballicohein, bbn, mistretta, bbn2}.

\begin{example}
    When $g=2$ and $d=2r+2$, Bhosle, Brambilla-Paz and Newstead \cite{bbn} observe that $H^0(X, \Omega\PP^r|_X \otimes \omega_X) \not=0$. Consequently, there is a nonzero morphism $T\PP^r|_X \to \omega_X$. Since both bundles have slope 2, we see that $T\PP^r|_X$ is not stable.
\end{example}

\begin{thm} \cite[Theorem 1.3]{farkaslarson}
    Let $X \subset \PP^r$ be a general BN-curve of genus $g \geq 2$. Then $T\PP^r|_X$ is stable unless $g=2$, $d=2r$ and $r \geq 3$. In the latter case, $T \PP^r|_X$ is strictly semistable.
\end{thm}

\subsection{Normal bundles of BN-curves}\label{sec-normal} In this section, we survey recent developments on the stability of normal bundles of general Brill--Noether curves. Our knowledge of the stability of normal bundles is less complete than our knowledge of the stability of restricted tangent bundles.

We call a triple of nonnegative integers $(g,r, d)$ a \defi{BN-triple} if $\rho(g,r, d)\geq 0$.  A BN-triple $(g,r, d)$ is  \defi{(semi)stable} if the general BN-curve of degree $d$ and genus $g$ in $\PP^r$ has a (semi)stable normal bundle. Otherwise, $(g,r, d)$ is \defi{unstable}. 

There are several examples where the general BN-curve does not have stable normal bundle.

\begin{example}[Canonical curves] \label{ex-canonical} The triples $(4,3, 6)$ and $(6,5, 10)$ corresponding to canonical curves of genus $4$ and $6$ are unstable, as we now explain. A canonical curve $X$ of genus $4$ is a $(2,3)$ complete intersection in $\PP^3$ and the normal bundle of the curve in the quadric destabilizes $N_{X/\PP^3}$. A general canonical curve $X$ of genus $6$ lies in a del Pezzo surface $S$ of degree $5$ and $N_{X/S}$ destabilizes $N_{X/\PP^5}$. The BN-triple $(5,4, 8)$ corresponding to a canonical curve of genus $5$ is semistable but not stable, since the general canonical curve $X$ of genus 5 is a $(2,2,2)$ complete intersection in $\PP^5$, hence $N_{X/\PP^5}= \OO_X(2)^{\oplus 3}$.
\end{example}

\begin{example}[Genus $2$ curves]\label{ex-genus2} The BN-triples $(2,3, 5)$, $(2,4, 6)$, $(2,5, 7)$, and $(2,6, 8)$ are unstable, as we now explain. Every genus $2$ curve $X$ is hyperelliptic and lies in a rational surface scroll $S$ swept out by the lines spanned by the fibers of the hyperelliptic map. The normal bundle $N_{X/S}$ is a line bundle of degree 12 which destabilizes $N_{X/\PP^r}$ for these 4 triples.
\end{example}

\begin{example}[Curves on quadrics]\label{ex-quadrics}
    Let $X$ be a general BN-curve of degree $7$ and genus $3$ in $\PP^4$. Then $X$ lies on a 3-dimensional vector space of quadrics and the residual is a trisecant line $\ell$. Let $\ell \cap X = \{p_1, p_2, p_3\}.$ Then \[N_{X/\PP^4} \cong \bigoplus_{i=1}^3 \OO_X(2)(-p_i),\] and is semistable but not stable.

    Similarly, let $X$ be a general BN curve in $\PP^4$ with $(g, d)= (6, 9)$ or $(7, 10)$. Then $X$ lies on a quadric hypersurface $Q$ and the quotient $N_{X/\PP^4} \to \OO_X(2)$ destabilizes the normal bundle. 
\end{example}

In the 1980s several authors investigated the stability of normal bundles of BN-curves in $\PP^3$ in low degree and genus.  The stability was proved for $(g, d) = (2, 6)$ by Sacchiero \cite{sacchiero}, for $(g, d) = (9, 9)$ by Newstead \cite{newstead}, for $(g, d) = (3, 6)$ by Ellia \cite{ellia}, and for $(g, d) = (5, 7)$ by Ballico and Ellia \cite{ballicoellia}.  Ellingsrud and Hirschowitz \cite{ellingsrudhirschowitz} announced a proof of stability of normal bundles of space curves in an asymptotic range of degrees and genera. This culminated in a full classification of stable BN-triples in $\PP^3$ by the authors.

\begin{thm}\cite[Theorem 1]{clv}
Let $X$ be a general BN-curve in $\PP^3$ of degree $d$ and genus $g \geq 2$. Then $N_{X/\PP^3}$ is stable except when $(g, d) \in \{ (2, 5), (4, 6) \}$.
\end{thm}

\begin{rem}
    Observe that we have encountered the two exceptional cases in $\PP^3$, namely $(2,3, 5)$ and $(4,3, 6)$, in Examples \ref{ex-genus2} and \ref{ex-canonical} respectively. 
\end{rem}

\noindent
Forthcoming work of Coskun, Jovinelly, and E.\ Larson \cite{coskunjovinellylarson} classifies stable BN-triples in $\PP^4$.

The stability of the normal bundle of a general canonical curve has also received considerable attention.  Aprodu, Farkas and Ortega \cite{AFO} conjectured that the normal bundle of a general canonical curve of genus at least 7 is stable and settled the case $g=7$. Bruns \cite{Bruns} proved the stability when $g=8$.  The authors characterized the semistable canonical triples $(g, g-1, 2g-2)$.

\begin{thm}\cite[Theorem 1.1]{clvcanonical}\label{thm-canonical}
Let $X$ be a general canonical curve of genus $g \not\in \{4, 6\}$. Then $N_{X/\PP^{g-1}}$ is semistable. In particular, if $g \equiv  1$ or $3 \pmod{6}$, then $N_{X/\PP^{g-1}}$ is stable. 
\end{thm}

\begin{rem}
   Let $H$ denote the hyperplane class on $\PP^g$. Since the general canonical curve arises as a hyperplane section of a K3 surface when $3 \leq g<10$ or $g=11$ (see \cite{mukai, mukai2}), Theorem \ref{thm-canonical} implies that the normal bundle of a general K3 surface in $\PP^g$ of degree $2g-2$ is $\mu_H$-semistable when $g \in \{3, 5, 7, 8, 9, 11\}$. 
\end{rem}

We 
We next illustrate a few of the key ideas that go into proving this theorem. The proof proceeds via degeneration to a reducible nodal curve  \(E \cup R\), where \(E\) is an elliptic curve of degree \(g\), and \(R\) is a rational curve of degree \(g - 2\) passing through the hyperplane section of \(E\).

By Lemma~\ref{lem:naive}, it suffices to show \(N_{E \cup R}|_E\) and \(N_{E \cup R}|_R\) are both semistable. This strategy can be made to work when \(g\) is odd, although when \(g\) is even, \(N_{E \cup R}|_R\) is unstable and so a more delicate analysis is required. For simplicity, we assume here that \(g = 6k + 3\) and focus on the restriction \(N_{E \cup R}|_E\). In this case we find specializations \(N_b\) of \(N_{E \cup R}|_E\) that fit into exact sequences
\[0 \to S_b \to N_b \to Q_b \to 0\]
whose slopes are
\begin{equation} \label{decsl}
\mu(S_b) = 6k + 6 + \frac{1}{k} > \mu(N_b) = 6k + 6 + \frac{6}{6k + 1} > \mu(Q_b) = 6k + 6 + \frac{5}{5k + 1}.
\end{equation}
Note that this implies \(N_b\) is unstable! Nevertheless these slopes are very close; indeed, they are adjacent  fractions in the Farey sequence with denominators bounded by \(6k + 1\). The key input that allows us to prove semistability from \emph{unstable} specializations is the following:

\begin{lem}\cite[Lemma 2.6]{clvcanonical}
    Let \(N_b\) for \(b \in B\) be a family of vector bundles parameterized by a rational base \(B\) on a smooth curve \(C\). Suppose that for \(b_1, b_2 \in B\), the specializations \(N_{b_i}\) fit into exact sequences
    \[0 \to S_i \to N_{b_i} \to Q_i \to 0\]
    with \(S_i\) and \(Q_i\) stable, of slopes given in \eqref{decsl}. If \(c_1(S_1) \neq c_1(S_2)\), the generic fiber of \(N\) is semistable.
\end{lem}
\begin{proof}[Proof sketch]
Our assumptions on \(S_i\) and \(Q_i\) imply that \(S_i\) is the unique destabilizing subbundle of \(N_{b_i}\). 
Hence, if the generic fiber of \(N\) were unstable, then its maximal destabilizing subbundle would have to specialize to \(S_i\) at \(b_i\).
The first Chern class of its maximal destabilizing subbundle would therefore give a nonconstant rational map \(B \dashedrightarrow \operatorname{Pic} C\), which is a contradiction.
\end{proof}

\begin{rem}
This argument works more generally; for a precise characterization of which slopes are allowable in terms of the Farey sequence see \cite[Lemma 2.6]{clvcanonical}.
\end{rem}

The appropriate specializations are constructed by suitably degenerating \(R\), in such a way that projection from a certain linear space yields the desired exact sequences.

\begin{problem}
    Strengthen Theorem \ref{thm-canonical} to show that when $g \geq 7$, then the general canonical curve has stable normal bundle.
\end{problem}

\begin{rem}
  Even when the general BN-curve has stable normal bundle, certain special BN-curves may have unstable normal bundles. For example, trigonal and tetragonal canonical curves always have unstable normal bundles, as we now explain. A trigonal canonical curve $X$ is contained in a rational normal surface scroll $S$. The normal bundle $N_{X/S}$ is a destabilizing subbundle of $N_{X/\PP^{g-1}}$. In characteristic 0, Fontana recently showed that this is the HN-filtration of $N_{X/\PP^{g-1}}$ for the general trigonal curve \cite{fontana}.

  Similarly, a tetragonal canonical curve $X$ lies on a rational normal threefold scroll $S$. The normal bundle $N_{X/S}$ destabilizes $N_{X/\PP^{g-1}}$. The curve $X$ is a complete intersection in $S$. When the genus is odd, the two divisor classes defining $X$ are the same and the normal bundle $N_{X/S}$ is semistable. Fontana shows that for the general tetragonal curve of odd genus, $N_{X/S} \subset N_{X/\PP^{g-1}}$ is the HN-filtration. When the genus is even, the two divisor classes defining $X$ are not equal and $X$ lies on a special surface $Y \subset S$ (this is analogous to the case of the del Pezzo surface in the case of $g=6$ in Example \ref{ex-canonical}). Then $N_{X/Y}$ destabilizes $N_{X/S}$. In characteristic 0, Fontana shows that for a general tetragonal curve of even genus, $N_{X/Y} \subset N_{X/\PP^{g-1}}$ is the HN-filtration \cite{fontana2}.
\end{rem}

\noindent
These examples raise the following problems.

\begin{problem}
Let \(X\) be a general curve of genus \(g\) and gonality \(4 < k < \lfloor (g+3)/2 \rfloor\).  Is \(N_{X/\PP^{g-1}}\) (semi)stable?
\end{problem}

\begin{problem}
Let \(X\) be a general hyperplane section of a K3 surface in \(\PP^g\) for \(g=10\) and \(g \geq 12\).  Is the normal bundle \(N_{X/\PP^{g-1}}\) (semi)stable?
\end{problem}

\begin{problem}
Determine the locus of canonical curves whose normal bundle is unstable.
\end{problem}

\noindent
More generally, one can raise the following questions. 

\begin{problem}
Prove that the normal bundle of a general $n$-canonical curve is stable (provided that $g$ and $n$ are sufficiently large). Determine the locus of $n$-canonical curves whose normal bundle is unstable. Given an arbitrary curve $X$ of genus $g$, does there exist $n \gg 0$ such that the normal bundle of the $n$-canonical embedding of $X$ is stable?
\end{problem}

A concept closely related to stability is interpolation. Roughly speaking interpolation asks for the maximum number of general points through which a given type of curve can pass. Via deformation theory, this question can be reformulated in terms of cohomological properties of the normal bundle. 

\begin{defin}
A vector bundle $V$ on a curve $X$ \defi{satisfies interpolation} if $H^1(X,V)=0$, and, for all $e>0$, there exists an effective divisor $D$ of degree $e$ on $X$ such that $H^0(X,V(-D))=0$ or $H^1(X,V(-D))=0$.
\end{defin}

Building on earlier work of Atanasov, E.\ Larson, and Yang \cite{aly}, E.\ Larson and Vogt classified when the normal bundle of a general BN-curve satisfies interpolation. 

\begin{thm}\cite[Theorem 1.2]{interpolation}\label{lvMain} The normal bundle of a general BN-curve of degree $d$ and genus $g\geq 1$ satisfies interpolation except when 
$$(g,r, d) \in \{(2,3, 5), (4,3, 6), (2,4, 6), (2,5, 7), (6, 5, 10)\}$$
\end{thm}

In general, interpolation and stability are distinct notions. However, interpolation implies semistability when the rank divides the degree. 

\begin{cor}[{Follows from \cite{interpolation}, cf.\ \cite[Remark 1.6]{vogt}}]\label{interpolation}
Let $(g,r, d)$ be a BN-triple such that $r-1$ divides $2d+2g-2$. Assume \[(g,r, d) \notin \{(2,3, 5), (4,3, 6), (2,5, 7)\}.\]
Then for the general BN-curve $X$ of type $(g,r, d)$, $N_{X/\PP^r}$ is semistable. 
\end{cor}
\begin{proof}
If $r-1$ divides $2d+2g-2$, then $\mu(N_{X/\PP^{r}})$ is an integer. Let $D$ be a divisor on $X$ of degree $\mu(N_{X/\PP^{r}})-g+1$ such that $H^0(X,N_{X/ \PP^{r}}(-D))=0$ or $H^1(X,N_{X/\PP^{r}}(-D))=0$. Since $\chi(X,N_{X/\PP^{r}}(-D))=0$, we must in fact have 
\[
H^0(X,N_{X/ \PP^{r}}(-D))=H^1(X,N_{X/\PP^{r}}(-D))=0.
\]
Let $V$ be a nonzero subbundle of $N_{X/\PP^r}$. Then we have $H^0(X, V(-D))=0$, so \[\chi(X, V(-D))=\rk(V) \cdot (\mu(V)-\mu(N_{X/ \PP^r}))\]
is nonpositive, and thus $\mu(V)\leq \mu(N_{X/\PP^r})$.
\end{proof}

\begin{rem}
 The exceptional cases $(2,4, 6)$ and $(6,5, 10)$ do not appear in the corollary because $r-1$ does not divide $2d+2g-2$.  
\end{rem}

\noindent
In view of these theorems, it is natural to conjecture the following.

\begin{conj}\cite[Conjecture 1.1]{clv}\label{conj-main}
All but finitely many BN-triples $(g, r, d)$ with $g \geq 2$  are stable.
\end{conj}

Ballico and Ramella \cite{br99} prove that if $g \geq 2r$ and $d \geq 2g+3r+1$, then the general BN-curve has a stable normal bundle.  Recently, Ran obtained different asymptotic (semi)stability results for the normal bundles of BN-curves in $\PP^r$ and more generally  normal bundles of curves in hypersurfaces \cite{ran23}. Coskun and Smith sharpened and generalized these results, providing some evidence for Conjecture \ref{conj-main}.

\begin{thm}\cite[Theorem 1.1]{coskunsmithnb}\label{largeMain}
Let $X\subset \PP^r$ be a general BN-curve of degree $d$ and genus $g \geq 2$. Let $u = \left\lfloor \frac{(r+1)^2}{2(r-1)} \right\rfloor$. Let $k$ be the integer such that $(k-1) u < g-1 \leq ku$.  Then $N_{X/\PP^r}$ is semistable if  one of the following holds:
\begin{enumerate}
        \item  $
    g\geq \binom{r-1}{2}+2+\left \lceil \frac{5r^2-7r}{2(r-1)}\right\rceil r(r-1),
    $ 
\item 

$d\geq \min\left(g+\frac{r^2}{4}+2r-3, (k+1)(r+1), (g-1)(2r-3)+r+1\right)$.
\end{enumerate}
\end{thm}

Let $b_2(r)$ be the least integer such that for all $d \geq b_2(r)$, there exists positive integers $d_1, d_2$ such that $d= d_1 + d_2$, $d_1, d_2 \geq r+1$, and $\gcd(r-1, 2d_1 + 1) =1$. If $p \geq 5$ is the smallest prime number  that does not divide $r-1$, then $b_2(r) \leq 2r+\frac{p-1}{2}$.  In general, $b_2(r)= 2r + O(\log(r))$ and  $b_2(r) \leq \frac{5r-3}{2}$. In fact, if $r \geq 1636$, then $b_2(r) \leq 2.01r + 2.015$. Then preserving the notation from Theorem \ref{largeMain}, we have:

\begin{thm}\cite[Theorem 1.2]{coskunsmithnb}\label{thm-introstab}
   Let $X\subset \PP^r$ be a general BN-curve of degree $d$ and genus $g \geq 2$.  Then $N_{X/ \PP^r}$ is stable if  one of the following holds:
\begin{enumerate}
    \item  $g \geq \binom{r-1}{2}+3+\left \lceil\frac{2(r+1)b_2(r)+3r^2-13r-2}{2(r-1)}\right \rceil r(r-1),$
\item $d\geq b_2(r) + \min\left(g+\frac{r^2}{4}+r-5, k(r+1), (g-1)(2r-3)\right)$.
\end{enumerate} 
\end{thm}

\noindent
In particular, we obtain the following corollary.
\begin{cor}
  For every $r$, there are finitely many BN-triples $(g,r,d)$ that are not (semi)stable.
\end{cor}

\begin{problem}
    Classify the BN-triples that are not (semi)stable. 
\end{problem}

One can ask more detailed questions about the geometry of the locus where $N_{X/\PP^r}$ is not stable. Close to nothing seems to be known about the following problem.

\begin{problem}
    Describe the locus of BN-curves whose normal bundle is not stable. Determine the dimension and the number of irreducible components of this locus and compute the generic HN-filtration for each irreducible component.
\end{problem}

We have so far concentrated primarily on the case when the ambient variety is $\PP^r$. When the ambient variety is not $\PP^r$, even the analogue of the Brill--Noether theorem is often not known.

\begin{problem}
    Determine when the normal bundles of curves in other homogeneous varieties or complete intersections in homogeneous varieties are stable. When they are not stable determine their HN-filtration.
\end{problem}

\noindent
Ran has obtained some partial results for curves on Fano hypersurfaces (see \cite{ran23b, ran24, ran25}).

We have only surveyed these problems when $X$ is a curve. Similar questions can be posed for higher dimensional varieties.

\begin{question}
    Let $f\colon \PP^s \to \PP^r$ be a general nondegenerate map. Is the normal bundle $N_f$ stable? 
\end{question}

\noindent
It is known that  $N_f$ is semistable when $f$ is the $d$-uple Veronese embedding, at least when the characteristic is 0 (see for example \cite[Theorem 1.1]{shang}).

\begin{question}
    Let $X$ be a general rational normal scroll in $\PP^r$. Is $N_{X/\PP^r}$ stable? One can ask the same question when $X$ is replaced by any other variety. It would be very interesting to settle this question for intensively studied classes of varieties such as K3 surfaces, abelian varieties, and smooth toric varieties. When the Picard rank of $X$ is large, one can also consider the question for different polarizations.
\end{question}

\section{Bundles associated to covers}\label{sec-covers}

In this final section, we will survey some recent results on natural bundles associated to finite coverings of curves.

\subsection{Tschirnhausen bundles} Let \(f\colon X \to Y\) be a finite map of degree \(r\) between nonsingular, irreducible, projective curves of genera \(g(X)=g\) and \(g(Y)=h\), defined over an algebraically closed field of characteristic 0 or larger than \(r\). The map \(f\colon X \to Y\) is called \defi{primitive} if the induced map on \'etale fundamental groups \(f_*\colon\pi_1^{\et}(X)\to \pi_1^{\et}(Y)\) is surjective.  In general, every finite map \(f \colon X \to Y\) factors uniquely into a primitive map \(f^\pr \colon X \to Z\) followed by an \'etale map \(f^\et \colon Z \to Y\), where \(Z\) corresponds to the image \(f_* \pi_1^\et(X)\) in  \(\pi_1^{\et}(Y)\).

The inclusion \(\OO_Y \to f_* \OO_X\) given by pullback of functions admits a splitting by \(\frac{1}{r}\) times the trace map.  Hence, \[f_* \OO_X = \OO_Y \oplus \DTs(f)\] for a vector bundle \(\DTs(f)\). The dual of \(\DTs(f)\) is called the \defi{Tschirnhausen bundle} \(\Tsch(f)\) associated to \(f\). The bundle \(\Tsch(f)\) plays a central role in the study of covers of curves and the geometry of the Hurwitz scheme. A natural question raised, and answered asymptotically, by \cite{deopurkarpatel} and \cite{landesmanlitt} is to determine when \(\Tsch(f)\) is stable.

 Since $f_* \OO_X = \OO_Y \oplus \Tsch(f)^{\vee}$, we have 
\[\rk(\Tsch(f))= r-1 \quad \mbox{and} \quad \deg(\Tsch(f))= \deg(f_* \OO_X).\]
Let $b$ be the degree of the ramification divisor of \(f \colon X \to Y\). By the Riemann--Roch Theorem,
\[1-g = \chi(\OO_X)= \chi(f_* \OO_X) = \deg(\Tsch(f)^{\vee}) + r(h-1).\] By the Riemann--Hurwitz Theorem, we conclude that \[\deg(\Tsch(f))= \frac{b}{2}.\]

\begin{rem}
    Observe that \(\Tsch(f^\et)\) is naturally a direct summand of \(\Tsch(f)\).  Moreover,  \[\deg(\Tsch(f^\et))=0 \quad \mbox{and} \quad \deg(\Tsch(f))>0 \ \mbox{when} \ f \ \mbox{is not \'etale.}\] 
Hence, if \(f \colon X \to Y\) has a nontrivial factorization \(f = f^\et \circ f^\pr\) into an \'{e}tale and primitive part,  then \(\Tsch(f)\) is unstable. Therefore, the (semi)stability problem for Tschirnhausen bundles is only interesting when \(f\) is either primitive or \'etale. 
\end{rem}

\subsubsection*{The \'etale case} If \(f \colon X \to Y\) is \'etale, then the Galois group \(\Gal(X/Y)\) acts on the fibers of \(f\) by permutation. Fixing a basepoint, this action realizes \(\Gal(X/Y)\) as a subgroup of the symmetric group \(\mathfrak{S}_r\). In this case, $\Tsch(f)$ is always semistable and whether it is stable depends on the representation theoretic properties of the action of $\Gal(X/Y)$.

\begin{prop}\cite[Proposition 1.3]{clvtschirnhausen}\label{prop-etale}
Let \(f\colon X \to Y\) be an \'etale cover of degree \(r\). Then \(\Tsch(f)\) is semistable. Furthermore, \(\Tsch(f)\) is stable if and only if the restriction of the standard representation of \(\mathfrak{S}_r\) to \(\Gal(X/Y)\) is irreducible.
\end{prop}

\subsubsection*{The primitive case} When the characteristic of the base field is 0 or greater than \(r\), then the Hurwitz space of primitive degree \(r\) genus \(g\) covers of \(Y\) is irreducible if the degree of the ramification divisor is positive, and otherwise empty. Hence, we can refer to the general degree \(r\) genus \(g\) primitive cover of \(Y\). 

\begin{thm}[{\cite{ballico89} for \(h=0\), and \cite[Theorem 1.4]{clvtschirnhausen} for \(h \geq 1\)}]\label{thm:tsch}
Let \(f \colon X \to Y\) be a general primitive degree \(r\) cover, where \(X\) has genus \(g\) and \(Y\) has genus \(h\), over an algebraically closed field of characteristic 0 or greater than \(r\).  
\begin{enumerate}
   \item\label{main_h1} If \(h=0\), then \(\Tsch(f)\) is  balanced.
   \item\label{main_h1} If \(h=1\), then \(\Tsch(f)\) is  semistable.
    \item\label{main_h2} If \(h\geq 2\), then \(\Tsch(f)\) is stable.
\end{enumerate}
\end{thm}

Theorem \ref{thm:tsch} shows that one can get a rational map from the Hurwitz scheme parameterizing primitive covers of $Y$ to the moduli space of stable bundles on $Y$. 
Little seems to be known about the base locus of this rational map.

\begin{problem}
    Describe the locus in the Hurwitz scheme parameterizing primitive covers $f\colon X \to Y$ such that $\Tsch(f)$ is unstable.
\end{problem}

 \subsection{Generic pushforwards} 
 We suppose again that the characteristic is \(0\) or greater than \(r\).
If instead of pushing forward $\OO_X$, we push forward a general line bundle (or more generally, a general vector bundle), the pushforward will no longer split. 
 Motivated by the study of the theta linear series on the moduli spaces of vector bundles on curves, Beauville in \cite{beauville} (see also \cite[Conjecture 6.5]{beauville2}) made the following conjecture: 

\begin{conj}[Beauville]\label{conj-main}
Let $f\colon X \to Y$ be a finite morphism between smooth irreducible projective curves, and let $V$ be a general vector bundle on $X$. Then $f_* V$ is stable if the genus of \(Y\) is at least \(2\) and semistable if the genus of \(Y\) is \(1\).
\end{conj}

Let $\HH_{r,g}(Y)$ denote the Hurwitz space parameterizing smooth connected degree $r$ genus $g$ covers of $Y$. In general, $\HH_{r,g}(Y)$ is reducible, with components corresponding to the \'etale part of the primitive-\'etale factorization.
More explicitly, the components correspond to subgroups of the (\'etale) fundamental group $\pi_1(Y)$ of index either equal to \(r\) if \(g = r(h - 1)\), or dividing \(r\) if \(g > r(h - 1)\).

\begin{thm}[{\cite{beauville} for \(h=0\), and \cite[Theorem 1.2]{clvbeauville} for \(h\geq 1\)}]\label{thm:gen}
Suppose that the characteristic is \(0\) or greater than \(r\).
Let \(Y\) be any smooth irreducible projective curve of genus \(h\).  Let \(f \colon X \to Y\) be a general morphism in any component of \(\HH_{r,g}(Y)\).  Let \(V\) be a general vector bundle of any degree and rank on \(X\).
\begin{enumerate}
    \item If \(h=0\), then \(f_*V\) is balanced. 
    \item If \(h=1\), then \(f_*V\) is semistable.
    \item If \(h\geq 2\), then \(f_*V\) is stable.
\end{enumerate}
\end{thm}

\begin{rem}
   In particular,  Beauville's conjecture holds for the general member of any component of the Hurwitz space $\HH_{r,g}(Y)$. However, Beauville's original conjecture predicting the stability of the general pushforward for {\em{every}} cover remains open.
\end{rem}

\begin{rem}
It may happen that for special $V$, the bundle $f_* V$ is not semistable. For example, we already saw that $f_* \OO_X$ has $\OO_Y$ as a direct summand. When the map $f$ is ramified, $\OO_Y$ destabilizes $f_* \OO_X$. This example raises the following problem.
\end{rem}

\begin{problem}
    For a general $f\colon X \to Y$ describe the locus of semistable bundles $V \in M_X(r,d)$ for which $f_*V$ is not semistable. 
\end{problem}

In this section, we have again concentrated only on the case of curves. One could ask analogues of these questions for higher dimensional varieties. 

\bibliographystyle{plain}

\begin{thebibliography}{ABCH}

\bibitem[ABBLT22]{abblt22}
J. Alper, P. Belmans, D. Bragg, J. Liang, and T. Tajakka, Projectivity of the moduli space of vector bundles on a curve, in {\it Stacks Project Expository Collection (SPEC)}, London Math. Soc. Lecture Note Ser., 480, Cambridge Univ. Press, (2022,) 90--125.

\bibitem[AR17]{AlzatiRe}
A. Alzati and R. Re, Irreducible components of Hilbert Schemes of rational curves with given normal bundle, Algebr. Geom.,  4 no. 1 (2017), 79--103.



\bibitem[AFO16]{AFO}
M. Aprodu, G. Farkas and A. Ortega, Restricted Lazarsfeld-Mukai bundles and canonical curves, in Development of moduli theory--Kyoto 2013, Advanced studies in pure mathematics, vol. 69, Mathematical Society of Japan (2016), 303--322.

\bibitem[ACGH85]{acgh}
E. Arbarello, M. Cornalba, P. A. Griffiths and J. Harris, Geometry of algebraic curves. Vol. I, Grundlehren
der Mathematischen Wissenschaften  267, SpringerVerlag, New York, 1985, xvi+386.


\bibitem[ALY19]{aly}      
 A. Atanasov, E. Larson and D. Yang,  Interpolation for normal bundles of general curves, Mem. Amer. Math. Soc., {\bf 257} no. 1234 (2019), v+105.  

 \bibitem[A57]{Atiyah}
M.F. Atiyah, Vector bundles over an elliptic curve, Proc. London Math. Soc., {\bf 7} (1957), 414--452.


\bibitem[BE84]{ballicoellia}
 E. Ballico and Ph. Ellia, Some more examples of curves in P3 with stable normal bundle, J. Reine Angew. Math. 350 (1984), 87--93.

\bibitem[Ba89]{ballico89}
E. Ballico, A remark on linear series on general k-gonal curves. 
Boll. Un. Mat. Ital. A (7) {\bf 3} no. 2 (1989), 195--197.

\bibitem[BH98]{ballicohein}
 E. Ballico and G. Hein, On the stability of the restriction of $T\PP^n$
to projective curves, Arch. Math. 71 (1998), 80--88.

\bibitem[BR99]{br99}
 E. Ballico and L. Ramella, On the existence of curves in $\PP^n$ with stable normal bundle, Ann. Polon. Math. {\bf 72} no.1
(1999), 33--42.

\bibitem[Bea00]{beauville}
A. Beauville, On the stability of the direct image of a generic vector bundle, preprint (2000).

\bibitem[Bea06]{beauville2}
A. Beauville, Vector bundles on curves and theta functions, Moduli Spaces and Arithmetic Geometry (Kyoto 2004),
Advanced Studies in Pure Mathematics 45 (2006), 145--156.


\bibitem[BBN08]{bbn}
U. Bhosle, L. Brambila-Paz and P. Newstead, On coherent systems of type $(n, d, n+1)$ on Petri curves, Manuscripta Math. 126 (2008), 409--441.

\bibitem[BBN15]{bbn2}
 U. Bhosle, L. Brambila-Paz and P. Newstead, On linear series and a conjecture of D.C. Butler, International Journal of Math. 26 (2015), 1550007, 18 pp.

\bibitem[Br13]{Bridges}
T. Bridges, R. Datta, J. Eddy, M. Newman and J. Yu, Free and very free morphisms into a Fermat hypersurface, Involve {\bf 6} (2013), no. 4, 437--445.

 \bibitem[B17]{Bruns}
G. Bruns, The normal bundle of canonical genus \(8\) curves, arXiv:1703.06213v2.


\bibitem[Cao]{cao}
A. Cao, Balanceness of normal bundles of rational curves in Grassmannian, in progress.



\bibitem[CZ14]{CZ14}
Q. Chen and Y. Zhu, Very free curves on Fano complete intersections, Algebr. Geom., \textbf{1} no. 5 (2014) 558--572.

\bibitem[Che25]{cheng}
R. Cheng, Free curves in Fano hypersurfaces must have high degree, Proc. Amer. Math. Soc. 153 (2025), 2841--2846.



 \bibitem[Cos08]{coskungw}
I. Coskun, Gromov-Witten invariants of jumping curves, Trans. Am. Math. Soc., 360 (2008), 989--1004.


 \bibitem[CH15]{chgokova}
I. Coskun and J. Huizenga, The birational geometry of the moduli spaces of sheaves on the plane, Proceedings of the G\"{o}kova Geometry-Topology Conference 2014, (2015), 114--155.

 \bibitem[CHN24]{chnsurvey}
I. Coskun, J. Huizenga and H. Nuer, The Brill-Noether Theory of the moduli spaces of sheaves on surfaces, in Moduli spaces and vector bundles-new trends, Contemp. Math. 803 (2024), 103--151.

\bibitem[CHS25]{coskunhuizengasmith}
I. Coskun, J. Huizenga and G. Smith, Stability and cohomology of Kernel bundles on projective
space, Michigan Math. J. 75 no. 1 (2025), 173--198.


\bibitem[CJL26]{coskunjovinellylarson}
I. Coskun, E. Jovinelly and E. Larson, Stability of normal bundles of Brill-Noether curves in $\PP^4$, in preparation.


\bibitem[CLV22]{clv}
I. Coskun, E. Larson and I. Vogt, Stability of normal bundles of space curves, Algebra Number Theory 16 no. 4 (2022), 919--953.

\bibitem[CLV24a]{clvbeauville}
I. Coskun, E. Larson and I. Vogt, Generic Beauville’s conjecture, Forum Math. Sigma 12
(2024), Paper No. e51, 7 pp.

\bibitem[CLV24b]{clvtschirnhausen}
I. Coskun, E. Larson and I. Vogt, Stability of Tschirnhausen bundles, Int. Math. Res. Not.
(IMRN) 2024 no. 1, (2024), 597--612.


\bibitem[CLV25]{clvgrassmannian}
I. Coskun, E. Larson and I. Vogt, Normal bundles of rational curves in Grassmannians, to appear in Algebra Number Theory

\bibitem[CLV25b]{clvcanonical}
I. Coskun, E. Larson and I. Vogt, The normal bundle of a general canonical curve of genus at
least 7 is semistable, to appear in J. Eur. Math. Soc. 

\bibitem[CR18]{coskunriedl}
I. Coskun and E. Riedl, Normal bundles of rational curves in projective space, Math. Z.  {\bf 288} (2018), 803--827.

\bibitem[CR19]{coskunriedlci}
I. Coskun and E. Riedl, Normal bundles of rational curves on complete intersections, Communications in Contemporary Math. 21 no. 2 (2019), (23 pages)

\bibitem[CS22]{coskunsmithvf}
I. Coskun and G. Smith, Very free rational curves on Fano varieties, J. Algebra 611 (2022), 246--264.

\bibitem[CS25]{coskunsmithnb}
I. Coskun and G. Smith, Stability of normal bundles of Brill-Noether curves, Math. Ann., 391 no. 4 (2025), 4997--5032.

\bibitem[DP22]{deopurkarpatel}
A. Deopurkar and A. Patel, Vector bundles and finite covers, Forum Math. Sigma 10 (2022), Paper No. e40, 30 pp.


\bibitem[EiL92]{einlazarsfeld}
L. Ein and R. Lazarsfeld, Stability and restrictions of Picard bundles, with an application to the normal bundles of elliptic curves. Complex projective geometry (Trieste, 1989/Bergen, 1989), 149--56, 
London Math. Soc. Lecture Note Ser., 179, Cambridge Univ. Press, Cambridge, 1992. 


\bibitem[EH83]{EH83}
D.~Eisenbud and J. Harris.
\newblock Divisors on general curves and cuspidal rational curves.
\newblock {\em Invent. Math.}, 74 no. 3 (1983), 371--418.

\bibitem[EH87]{Eishar}
D. Eisenbud and J. Harris, Irreducibility and monodromy of some families of linear series, Ann. Sci. \'{E}cole Norm. Sup. (4) 20 no. 1 (1987), 65--87.


\bibitem[EV81]{eisenbudvandeven}
D. Eisenbud and A. Van de Ven, On the normal bundles of smooth rational space curves, Math. Ann., 256 (1981), 453--463.

\bibitem[EV82]{eisenbudvandeven82}
D. Eisenbud and A. Van de Ven, On the variety of smooth rational space curves with given degree and normal bundle, Invent. Math., 67 (1982), 89--100.


\bibitem[E83]{ellia}
Ph. Ellia, Exemples de courbes de $\PP^3$ a fibre normal semi-stable, stable, Math. Ann. 264 (1983), 389–396.

\bibitem[ElH84]{ellingsrudhirschowitz}
G. Ellingsrud and A. Hirschowitz, Sur le fibr\'{e} normal des courbes gauches, C. R. Acad. Sci. Paris S\'{e}r. I Math. 299 no. 7 (1984), 245--248.

\bibitem[ElL81]{ellingsrudlaksov}
G. Ellingsrud, D. Laksov, The normal bundle of elliptic space curves of degree five, in 18th Scand. Congress of Math. Proc. 1980, ed. E. Balslev, Boston 1981, 258--287.

\bibitem[FL25]{farkaslarson}
G. Farkas and E. Larson, The minimal resolution property for points on general curves, Ann. Sci. \'{E}c. Norm. Sup\'{e}r., 58 no. 2 (2025), 433-462.

\bibitem[Fo25a]{fontana}
H. Fontana, The Harder-Narasimhan filtration of a trigonal canonical curve, preprint, arXiv:2505.14907



\bibitem[Fo25b]{fontana2}
H. Fontana, Tetragonal canonical curves, preprint, arXiv:2512.16191 

\bibitem[GS80]{ghionesacchiero}
F. Ghione and G. Sacchiero, Normal bundles of rational curves in $\PP^3$, Manuscripta Math., 33 (1980), 111--128.

\bibitem[Gi82]{gieseker}
D. Gieseker, Stable curves and special divisors: Petri’s conjecture, Invent. Math. 66 no. 2 (1982), 251--275.

\bibitem[GH80]{griffithsharris}
P. Griffiths and J. Harris,  On the variety of special linear systems on a general algebraic curve, Duke Math.
J., 47 no. 1 (1980), 233--272.

\bibitem[Ha10]{hartshorne}
R. Hartshorne, Deformation Theory, Graduate Texts in Mathematics 257, Springer, New York, N.Y., 2010. 

\bibitem[HH85]{hartshornehirschowitz}
R. Hartshorne and A. Hirschowitz, Smoothing algebraic space curves, Algebraic geometry, Sitges (Barcelona), 1983, Lecture Notes in Math., 1124, Springer, Berlin (1985), 98--131.

\bibitem[Ho10]{hoffman}
N. Hoffmann, Moduli stacks of vector bundles on curves and the King-Schofield rationality proof in Cohomological and geometric approaches to rationality problems, 133--148, Progr. Math., 282, Birkh\"{a}user
Boston, Boston, MA, 2010.

\bibitem[HL10]{huybrechtslehn}
D. Huybrechts and M. Lehn, The geometry of moduli spaces of sheaves Cambridge University Press, Cambridge, 2010.

\bibitem[JP24]{jensenpayne}
D. Jensen and S. Payne, Tropical independence I: Shapes of divisors and a proof of the Gieseker–Petri theorem, Algebra Number Theory, 8 no. 9 (2014), 2043--2066.





\bibitem[K96]{Kol96}
J. Koll\'{a}r, Rational curves on algebraic varieties, Springer, 1996.

\bibitem[LL24]{landesmanlitt}
A. Landesman and D. Litt, Geometric local systems on very general curves and isomonodromy, J. Amer. Math. Soc. 37 (2024), 683--729.


\bibitem[L24]{la24}
E. Larson, On the cohomology of $N_C(-2)$ in positive characteristic,  Commun. Contemp. Math.,  26 (2024), no. 10, Paper No. 2350067, 11 pp.

\bibitem[L21]{hannah}
H. Larson, Normal bundles of lines on hypersurfaces,  Michigan Math. J., 70
no. 1 (2021),  115--131.

\bibitem[LV23]{interpolation}
E. Larson and I. Vogt, Interpolation for Brill-Noether curves, Forum Math. Pi, vol. 11 (2023), e25.

\bibitem[LP97]{lepotierbook}
J. Le Potier, Lectures on vector bundles, translated by A. Maciocia, Cambridge Studies in Advanced
Mathematics, 54, Cambridge Univ. Press, Cambridge, 1997.

\bibitem[Ma19]{Ma19}
S. Mandal,  On the loci of morphisms from $\PP^1$ to 
$G(r, n)$ with fixed splitting type of the
restricted universal sub-bundle or quotient bundle, J. Algebra,  585  (2021), 759--783.

\bibitem[Mi25a]{mioranci}
L. Mioranci, Normal bundles of rational normal curves on hypersurfaces, 
Michigan Math. J., 75 no. 3 (2025), 639--671.


\bibitem[Mi25b]{mioranci2}
L. Mioranci, Restricted tangent bundle of rational curves on projective hypersurfaces, preprint, arXiv:2507.13927

\bibitem[Mis08]{mistretta}
 E. Mistretta, Stability of line bundle transforms on curves with respect to low codimensional subspaces, J. London Math. Soc. 78 (2008), 171--182.

\bibitem[Mu92]{mukai}
 S. Mukai, Fano 3-folds, in Complex Projective Geometry, London Math. Soc. Lecture Notes Ser. 179, Cambridge University Press (1992), 255--263.

\bibitem[Mu96]{mukai2}
S. Mukai, Curves and K3 surfaces of genus eleven, in Moduli of vector bundles, Lecture Notes in Pure and Appl. Math. 179, Dekker (1996), 189--197.

\bibitem[N83]{newstead}
P. E. Newstead, A space curve whose normal bundle is stable, J. London Math. Soc. (2) 28 (1983), 428--434.

\bibitem[O14]{osserman}
B. Osserman, A simple characteristic-free proof of the Brill–Noether theorem, Bull. Brazilian Math. Soc, 45 no. 4 (2014), 807--818.

\bibitem[Ram90]{ramella}
L. Ramella, La stratification du sch \'ema de Hilbert des courbes rationnelles de $\PP^n$ par le fibr\'e tangent restreint, C. R. Acad. Sci. Paris S\'er. I Math., 311 no. 3 (1990), 181--184.

\bibitem[R07]{ran}
Z. Ran, Normal bundles of rational curves in projective spaces, Asian J. Math., 11 no. 4 (2007), 567--608. 





\bibitem[R21]{ran21}
Z. Ran,  Low-degree rational curves on hypersurfaces in projective spaces and their fan degenerations, J. Pure Appl. Algebra 225 (2021), no. 2, Paper No. 106492, 20 pp.


\bibitem[R23a]{ran23}
Z. Ran, Balanced curves and minimal rational connectedness on Fano hypersurfaces,
Int. Math. Res. Not. IMRN 2023, no. 6, 4555--4600.

\bibitem[R23b]{ran23b}
Z. Ran, Curves with semistable normal bundle on Fano hypersurfaces, preprint (2023), arXiv:2211.12661.



\bibitem[R24]{ran24}
Z. Ran, Interpolation of curves on Fano hypersurfaces,
Commun. Contemp. Math. 26 (2024), no. 1, Paper No. 2350002, 41 pp.

\bibitem[R25]{ran25}
Z. Ran, Polarized interpolation and normal postulation for curves on Fano hypersurfaces, Ukrainian Math. J. 77 no.1  (2025), 105--115.



\bibitem[S80]{sacchiero80}
G. Sacchiero, Fibrati normali di curvi razionali dello spazio proiettivo, Ann. Univ. Ferrara Sez. VII, 26 (1980), 33--40.

\bibitem[S82]{sacchiero82}
G. Sacchiero, On the varieties parameterizing rational space curves with fixed normal bundle, Manuscripta Math., 37 (1982), 217--228.  





\bibitem[S83]{sacchiero}
G. Sacchiero, Exemple de courbes de $\PP^3$ de fibr\'{e} normal stable, Comm. Algebra, 11 no. 18 (1983), 2115--2121.

\bibitem[Sha24]{shang}
R. Shang, Slope semistability of Veronese normal bundles, preprint, arXiv:2411.16664

\bibitem[She12a]{She12}
M. Shen, Rational curves on Fermat hypersurfaces, C. R. Math. Acad. Sci. Paris {\bf 350} (2012), no.~15-16, 781--784.

\bibitem[She12b]{Shen2}
M. Shen, On the normal bundles of rational curves on Fano 3-folds, Asian J. Math. {\bf 16} no. 2 (2012), 237--270.

\bibitem[ST18]{ST18}
J. Starr and Z. Tian, Separable rational connectedness and weak approximation in positive characteristic, preprint, arXiv:1907.07041.

\bibitem[STZ18]{STZ18}
J. Starr, Z. Tian, and R. Zong, Weak approximation for Fano complete intersections in positive characteristic, preprint, arXiv:1811.02466.

\bibitem[Ti15]{Tia15}
Z. Tian, Separable rational connectedness and stability, in Rational points, rational curves, and entire holomorphic curves on algebraic varieties, Contemp. Math., \textbf{654} (2015), 155--160.

\bibitem[V18]{vogt}
I. Vogt, Interpolation for Brill-Noether space curves, Manuscripta Math. {\bf 156} (2018), 137--147.


\bibitem[Zh11]{Zhu}
Y. Zhu, Fano hypersurfaces in positive characteristic, preprint, arXiv:1111.2964.

\end{thebibliography}

\end{document}